\documentclass[10pt,reqno,final]{amsart}

\usepackage{epsfig,amssymb,amsmath,version}
\usepackage{amssymb,version,graphicx,fancybox,mathrsfs,multirow}
\usepackage{epstopdf}
\usepackage{url,hyperref}
\usepackage[notcite,notref]{showkeys}
\usepackage{bm}
\usepackage{subfigure}
\usepackage{color,xcolor}
\usepackage{cases}
\usepackage{mathtools}
\usepackage{extarrows}
\usepackage{bbm}
\usepackage[bb=ams, cal=cm, scr=boondox, frak=euler]{mathalpha}

\catcode`\@=11 \theoremstyle{plain}
\@addtoreset{equation}{section}   % Makes \section reset 'equation' counter.

\@addtoreset{figure}{section}
\renewcommand\thefigure{\thesection.\@arabic\c@figure}
\renewcommand{\thefigure}{\arabic{section}.\arabic{figure}}
\newtheorem{thm}{\bf Theorem}

\newenvironment{theorem}{\begin{thm}} {\end{thm}}
\newtheorem{cor}{\bf Corollary}

\newtheorem{lmm}{\bf Lemma}

\newenvironment{lemma}{\begin{lmm}}{\end{lmm}}

\newtheorem{example}{\bf Example}[section]

\newtheorem{assumption}{\bf Assumption}[section]

\theoremstyle{remark}
\newtheorem{rem}{\bf Remark}[section]
\theoremstyle{definition}

\numberwithin{table}{section}

\renewcommand \wedge \times

\begin{document}
\title{\bf{}}
\title[Strong convergence rates of a fully discrete scheme for Nonlinear SPDEs
	] {Strong convergence rates of a fully discrete scheme  for a class of nonlinear stochastic PDEs with non-globally Lipschitz coefficients driven by multiplicative noise}
		\author[
	C. Huang   \; $\&$ \; J. Shen
	]{
		\;\; Can Huang${}^1$ \;\; and\;\; Jie Shen${}^{2}$
		}
	
	\thanks{${}^1$School of Mathematical Sciences, Xiamen university. %and Fujian Provincial Key Laboratory on Mathematical Modeling \& High Performance Scientific Computing, Xiamen University, Fujian 361005, China. 
	The research of the first author is partially supported by NSFC grants No. 91630204 and No. 11771363. \\%and Fundamental Research Funds for the Central Universities (No. 20720180001). \\
		\indent ${}^{2}$Corresponding author. Department of Mathematics, Purdue University. The research of the second author is partially supported
by NSF grant  DMS-2012585 and AFOSR FA9550-20-1-0309.  \\
		%\indent The authors would like to thank both universities for hosting their mutual visits to complete this work. 
			}
	
\begin{abstract}
We consider  a fully discrete scheme   for nonlinear stochastic PDEs with non-globally Lipschitz coefficients driven by multiplicative noise in a multi-dimensional setting. Our method uses a polynomial based spectral method in space, so it  does not require the elliptic operator $A$ and the covariance operator $Q$ of noise in the equation commute, and thus successfully alleviates a restriction of Fourier spectral method for SPDEs pointed out by Jentzen, Kloeden and Winkel  in \cite{JentzenKW11}.
 The discretization in time is a tamed semi-implicit scheme which treats the nonlinear term explicitly while being unconditionally stable.  
Under regular assumptions which are usually made for SPDEs with additive noise, we establish optimal strong convergence rates in both space and time for our fully discrete scheme.  We also present numerical experiments which are consistent with  our theoretical results.    
\end{abstract}
	
\keywords{stochastic PDE,  spectral method, optimal, convergence rate}

 \subjclass[2010]{65N35, 65E05, 65N12,  41A10, 41A25, 41A30, 41A58}	
	
\maketitle

\section{Introduction}
We consider  numerical approximation of the following nonlinear stochastic  PDE perturbed by multiplicative  
 noise:
 \begin{equation}\label{model:AC}
\begin{cases}
\displaystyle{du=Audt+F(u)dt+G(u)dW^Q(t), \;\; x\in \mathcal{O}\subset \mathbb{R}^d \ (d=1,2)},\\
u(t, x)=0, x\in\partial \mathcal{O},\\
u(0,x)=u_0(x),\;\; x\in \mathcal{O},
\end{cases}
\end{equation}
 %in a separable Hilbert space $H(\mathcal{O})$ with usual norm $\|\cdot\|$. Here,
 where $A$ is the Laplacian operator on $\mathcal{O}$, 
 $F$ is the Nemytskii operator defined by $F(u)(\xi)=f(u(\xi)),\ \xi\in\mathcal{O}$, where
 $f$ is an odd-degree polynomial  with negative leading coefficient satisfying Assumption \ref{assump:2}. In particular, if $f(u)=u-u^3$, the equation becomes the well-known stochastic Allen-Cahn equation. $G(u)(\xi)=g(u(\xi))$ is another Nemytskii operator, where $g(u)$ is a Lipschitz continuous function with linear growth satisfying Assumption \ref{assump:3}, and $W^Q(t)$ is a Q-Wiener process on the probability space $(\Omega, \mathcal{F}, \{\mathcal{F}_t\}_{t\geq 0}, \mathbf{P})$ defined by (cf. \cite{PratoD96})
 \begin{align*}
 W^Q(t)=\sum\limits_{j=1}^\infty \sqrt{q}_j e_j\beta_j(t),
 \end{align*}
 where $\beta_j(t)$ are independent standard Wiener processes, and $\{(q_j, e_j)\}_{j=1}^\infty$ are eigen-pairs of a symmetric non-negative operator $Q$. We emphasize that $\{e_j\}_{j=1}^\infty$ are not necessarily eigenfunctions of $A$ in $\mathcal{O}$. 
 
  It is well known that for $u_0\in C(\mathcal{O})$,  \eqref{model:AC} admits a unique mild solution in $L^p(\Omega; C((0,T); H)\cap L^\infty(0,T; H))$ for arbitrary $p\geq 1$, that satisfy (cf. \cite{Cerrai03})
 \begin{align}\label{eq:truesol}
u(t)&=e^{tA}u_0+\int_0^t e^{(t-s)A}f(u(s))ds+\int_0^t e^{(t-s)A}g(u(s))dW^Q(s).
\end{align}
Moreover, under certain conditions to be specified later,  $\mathbf{E}\sup\limits_{t\in [0,T]}\|(-A)^{\frac{\gamma}{2}}u(t)\|^2<\infty$ for some $\gamma>1$ (cf. Theorem \ref{prop:H1r} below).

Many mathematical models in physics, biology, chemistry etc. are formulated as SPDEs (cf. \cite{Chow07, DaPrato91, LiuR05}),  and various  numerical methods have been proposed for solving SPDEs. We refer to \cite{AllenZ98, JentzenKW11, Kruse14, Wang17, Yan05} and references therein for an incomplete account of numerical approaches for SPDEs with global Lipschitz condition on $f$. In contrast,  SPDEs with non-globally Lipschitz condition on $f$ are more difficult to deal with, we refer to   \cite{BeckerGJK17, BeckerJ19, BrehierCH19, CuiH19,FengLZ17, KovacsLL15, KovacsLL18, MajeeP18, LordPS14, QiW19} for some recent advances in this regard. Moreover,  most of these work  are concerned  with additive noise (cf. \cite{BeckerGJK17, BrehierCH19, Debussche10,JentzenKW11, JentzenR18, Jentzen20, CuiH19, KovacsLL15, QiW19, Wang17}), while SPDEs with local Lipschitz condition driven by multiplicative noise have received much less attention.  We would like to point out  that in \cite{FengLZ17}, the authors considered a finite element method (FEM) for stochastic Allen-Cahn equation driven by the gradient type multiplicative noise under sufficient spatial regularity assumptions, and  in \cite{MajeeP18}, the authors also investigated a FEM for the same equation perturbed by multiplicative noise of type $g(u)\beta(t)$, where $\beta(t)$ is a Brownian motion.  In both cases,   fully implicit time discretization schemes are used  so that a nonlinear system has to be solved at each time step.   

The main goal of this paper is  to design and analyze  a  strongly convergent, linear and fully decoupled  numerical method for SPDEs with local Lipschitz condition and driven by multiplicative noise in a multi-dimensional framework.
To avoid using a fully implicit scheme for  SPDEs with local Lipschitz condition, we construct a tamed semi-implicit scheme in time (cf. \cite{ HutzenhalerJK12, TretyakovZ13, Gyongy16, Wang18}) and show  (cf.  Theorem \ref{thm:uncond}) that it is unconditionally stable under a quite general setting, which includes in particular the stochastic Allen-Cahn equation with multiplicative noise.
% making it essentially different from explicit methods \cite{HutzenhalerJK12} and other semi-implicit or implicit schemes \cite{QiW19, KovacsLL15}.
On the other hand, we adopt as spatial discretization a spectral-Galerkin method.
 Distinguished for their high resolution and relative low computational cost for a given accuracy threshold,  spectral methods have become a major computational tool for solving PDEs. However, only limited attempts have been made for using spectral methods 
  for SPDEs (cf. \cite{BeckerGJK17, JentzenKW11, JentzenR18}), and most of these attempts are confined to Fourier spectral methods. Note that the use of Fourier-spectral methods in these work is essential as  Fourier basis functions are eigenfunctions of the elliptic operator $-A$.  
  % Thus, it leads to a restriction pointed out  by Jentzen, Kloeden and Winkel \cite{JentzenKW11}:
%{\em The eigenfunctions of the dominating linear operator and of the covariance operator of the deriving additive noise process of the SPDE must coincide and must be known explicitly}. 
Since in our tamed semi-implicit scheme, the nonlinear term and the noise terms are treated explicitly,  we shall employ a polynomial-based spectral method  for spatial approximation to overcome the restriction mentioned above. 
A key ingredient is to use a set of specially constructed Fourier-like {it discrete} eigenfunctions of $A$ (cf.  \cite[Chapter 8]{ShenTW11}), which are mutually orthogonal in both $L^2(\mathcal{O})$ and  $H^1(\mathcal{O})$. 

Combining the above ingredients together, we develop a fully discretized scheme that can be fully decoupled, making it very efficient.
Moreover, our method yields the following  convergence rate  under regular assumptions (Assumption \ref{assump:1}-Assumption \ref{assump:4}):
\begin{align}
{ \mathbf{E}\|u(t_k)-u_N^{k}\|\leq C(N^{-\gamma}+\tau^{\frac{1}{2}}),  }
 \end{align}
 where $u_N^{k}$ is the full-discretization of $u$ at $t_{k}$, $N$ is the number of points  in each direction in our spatial approximation, $\tau$ is the time step size and $\gamma$ is the index measuring the regularity of noise, which can be arbitrarily large provided that Assumptions \ref{assump:3} and \ref{assump:4} hold.  It extends the  results  in \cite{CuiH19, QiW19} for stochastic Allen-Cahn equation with additive noise with finite-element approximation under the essential assumption $\|(-A)^{\frac{\gamma-1}{2}}Q^{\frac{1}{2}}\|_{L^2}<\infty$.

 In summary,  the main contributions of this paper include: 
 \begin{itemize}
 \item We investigate the optimal spatial regularity of solution for \eqref{model:AC}, which lifts  the previous results $\gamma\in (1,2]$ (cf. \cite{CuiH19, JentzenKW11, JentzenR18, Kruse14, Wang17, Wang18, QiW19}) to possible arbitrarily large $\gamma$ provided Assumption \ref{assump:1}-Assumption \ref{assump:4} are fulfilled, 
 and derive optimal spatial convergence rate for our fully-discretized scheme based on the improved regularity. 
 \item Our tamed time discretization  for \eqref{model:AC} treats the nonlinear terms explicitly while  is still unconditionally stable. Thus, it avoids solving  nonlinear systems at each time  step, which is in contrast to the popular backward Euler method  (cf. \cite{JentzenKW11, JentzenR18, Kruse14, Wang18, QiW19}). %It is remarkable to point out that our scheme can also be applied to solve stochastic Allen-Cahn equation, which is comparable to operator splitting method in this regard (cf. \cite{BrehierCH19, CuiH19}).   
 \item We use the Legendre spectral method, instead of the usual Fourier spectral method,  for spatial discretization which  does not require  the commutativity of operators  $A$ and $Q$, and circumvents a restriction of Fourier approximation for SPDE pointed out in \cite{JentzenKW11}. Through a matrix diagonalization process, our method based on the Legendre approximation  can also be efficiently implemented as with a Fourier approximation.
 \end{itemize}
 
The rest of this paper is organized as follows. In Section 2, some preliminaries including our main assumptions and optimal spatial regularity of solution of \eqref{model:AC} are presented. Section 3 is devoted to spatial semi-discretization and its analysis. 
In Section 4, we present our semi-implicit tamed Euler full-discretization for \eqref{model:AC}, and derive optimal convergence rate for the scheme  under regular assumptions.  In Section 5, we present numerical results  for the stochastic Allen-Cahn equation  to validate our main theoretical results. 

\section{Preliminaries}
In this section, we first describe some notations and a few lemmas which will be used in our analysis, and then we present several general assumptions  for the problem under consideration. 
\subsection{Notations}
We begin with notations.  Let $U$ and $V$ be separable Hilbert spaces.  We denote the norm in $L^p(\Omega, \mathcal{F}, \mathbf{P}; U)$ by $\|\cdot\|_{L^p(\Omega; U)}$, that is,
$$
\|Y\|_{L^p(\Omega;U)}=\big(\mathbf{E}\big[\|Y\|_U^p\big]\big)^{\frac{1}{p}},\; \; Y\in L^p(\Omega, \mathcal{F}, \mathbf{P}; U).
$$
Denote by $L_1(U,V)$ the nuclear operator space from $U$ to $V$ and for $T\in L_1(U,V)$, its norm is given by
$$
\|T\|_{L_1}=\sum\limits_{i=1}^\infty |(Te_i,e_i)_U|\ \;\text{and}\;\; Tr(T)=\sum\limits_{i=1}^\infty (Te_i,e_i)_U
$$
for any orthonormal basis $\{e_i\}$ of $U$. In particular, if $T>0$, then $\|T\|_{L^1}=Tr(T)$. In this work, we assume that $W^Q(t)$ is of trace class, i.e.
$Tr(Q)<\infty$.  Let $L_2(U,V)$ be the Hilbert-Schmidt space such that for any $T\in L_2(U,V)$
$$
\|T\|_{L_2}=\bigg(\sum\limits_{i=1}^\infty\|Te_i\|^2\bigg)^{1/2}<\infty.
$$
Moreover, if $Q$ is of trace class, we introduce $L_2^0=L_2(U_0, V)$ with norm 
$$
\|T\|_{L_2^0}=\|TQ^{1/2}\|_{L_2(U,V)},
$$
 where $U_0=Q^{1/2}(U)$.

 The following properties are frequently used hereafter 
 $$
 \|ST\|_{L_2}\leq \|S\|_{L_1}\|T\|_{L_2}, \; \|TS\|_{L_2}\leq \|T\|_{L_2}\|S\|_{L_1}\ \; S\in L_1(U,V),\ T\in L_2(U,V). 
 $$
Finally, when no confusion arises, we will drop the spatial dependency from the notations, i.e., $u(t)=u(t,x)$.
\subsection{Some useful lemmas}
We start with the Burkhold-Davis-Gundy-type inequality, which is a generalization of Ito isometry. 
{\begin{lemma}\cite[Theorem 6.1.2]{LiuR05}
For any $p\geq 2, 0\leq t\leq T$, and for any predictable stochastic process $\Phi(\sigma)$ which satisfies
\begin{align*}
\int_0^T\bigg(\mathbf{E}\|\Phi(\sigma)\|_{L_2^0}^p\bigg)^{2/p}d\sigma<\infty,
\end{align*}
we have
\begin{align*}
\bigg(\mathbf{E}\bigg(\sup\limits_{t\in [0,T]}\bigg\|\int_{0}^{t}\Phi(\sigma)dW(\sigma)\bigg\|^p\bigg)\bigg)^{\frac{1}{p}}\leq C(p)\bigg(\int_0^T\bigg(\mathbf{E}\|\Phi(\sigma)\|_{L_2^0}^p\bigg)^{2/p}d\sigma\bigg)^{\frac{1}{2}},\end{align*}
where $\displaystyle{C(p)=p\bigg(\frac{p}{2(p-1)}}\bigg)^{\frac{1}{2}}$. 
\end{lemma}
}

We recall the following generalized Gronwall's inequality and its discretized version:
%\begin{lemma} (Young's inequality for convolution \cite[Page 34]{Adams75})
%If $1/p+1/q=1+1/r$, and if $u\in L^p(\mathbb{R}^n), v\in L^q(\mathbb{R}^n)$, then $u\ast v\in L^r(\mathbb{R}^n)$ and 
%$$
%\|u\ast v\|_{L^r}\leq \|u\|_{L^p}\|v\|_{L^q}. 
%$$
%\end{lemma}
\begin{lemma}(Generalized Gronwall's lemma \cite{DixonM86})
Let $T>0$ and $C_1,C_2\geq 0$ and let $\phi$ be a nonnegative and consitnuous function. Let $\beta>0$. If we have
$$
\phi(t)\leq C_1+C_2\int_0^t (t-s)^{\beta-1}\phi(s)ds,
$$
then there exists a constant $C=C(C_2,T,\beta)$ such that
$$
\phi(t)\leq CC_1.
$$
\end{lemma}
%We also recall  the following generalized discretize Gronwall's inequality: 
%{
%\begin{lemma} (Generalized discrete Gronwall's lemma \cite{DixonM86})
% Let $t_i=ih, i=0,1,\cdots, M$ and let $v_i$ be a sequence of non-negative real numbers satisfying
%$$
%v_j\leq A+Bh\sum\limits_{i=0}^{j-1} t_{j-i}^{\beta-1}v_i,\;\; 1\leq j\leq M, \ \beta\in (0,1],
%$$
%where $A, B>0$ is independent of $h$ ($Mh\leq T$). Then, there exists a constant $C(A,B,T,\beta)$ such that
%$$
%v_j\leq C(A,B,T,\beta)\quad \forall 1\le j\le M. 
%$$
%\end{lemma}
%}

\subsection{Assumptions}
We describe below our main assumptions.
\begin{assumption}\label{assump:1}
(Operator $A$) The linear operator $-A:dom(A)\subset H\to H$ is densely defined, self-adjoint and positive definite with compact inverse.
\end{assumption}

Under this assumption, the operator $A$ generates an analytic semigroup $E(t)=e^{tA}, t\geq 0$ on $H$ and  the fractional powers of $(-A)$ and its domain $H^{r}:=\text{dom}((-A)^{r/2})$ for all $r\in \mathbb{R}$ equipped with inner product $(\cdot,\cdot)_r=((-A)^{r/2}\cdot, (-A)^{r/2}\cdot)$ and 
the induced norm $\|\cdot\|_r=(\cdot,\cdot)_r^{1/2}$. In particular, we denote $\|\cdot\|=\|\cdot\|_0$.
 Let $L_{2,r}^0=L_2(U_0, H^r)$ with norm $\|T\|_{L_{2,r}^0}=\|(-A)^{r/2}T\|_{L_2^0}$.
Moreover, the following inequalities holds (cf. \cite[Theorem 6.13]{Pazy83}, \cite{Kruse14}).
\begin{align}
(i)&\   \text{For\ any}\ \mu\geq 0, \ \text{it\ holds\ that}\qquad\qquad\qquad\qquad\qquad\qquad\qquad\qquad\qquad \qquad\qquad\notag\\
&\qquad\qquad\qquad (-A)^\mu E(t)v=E(t)(-A)^\mu v,\ \ \text{for} \ v\in {H}^{2\mu},\notag\\
&\ \text{and\ there\ exists\ a \ constant\ C\ such \ that}\notag\\
&\qquad\qquad\qquad \|(-A)^\mu E(t)\|\leq Ct^{-\mu},\ \ t>0; \label{eq:ineq1}\\
(ii)&\ \text{For\ any}\ 0\leq\nu\leq 1, \ \text{there\ exists\ a\ constant\ C\ such\ that}\notag\\
&\qquad\qquad\qquad \|(-A)^{-\nu}(E(t)-I)\|\leq Ct^\nu,\ t>0. \label{eq:ineq2}
\end{align}
\
\begin{assumption} \label{assump:2}
(Nonlinearity) Let $F(v)(x)$ be a Nemytskii operator defined by $F(v)(x):=f(v(x))=\sum\limits_{j=0}^Pa_j (v(x))^j, a_j\in\mathbb{R} $ where $P$ is an  odd integer with $a_P<0$ such that %the following conditions hold (cf. \cite{CuiH19,LiuR05}):
the following coercivity and one-sided Lipschitz condition hold
\begin{align}\label{fcoer}
&\langle f(u),u\rangle \leq -\theta \|u\|^\alpha+K\|u\|^2,
\quad \text{ for some }\; \theta, {\alpha}, K>0; \notag\\
&\langle f(u)-f(v), u-v\rangle\leq L\|u-v\|^2, \quad L>0, \ u, v\in L^{2P}(\mathcal{O}). 
%& |f(\xi)|\leq L_f(1+|\xi|^P),\quad\quad |f‘(\xi)|\leq L_f(1+|\xi|^{P-1}),
\end{align}
for some $L>0$. 
\end{assumption}

%Assumption \ref{assump:2} ensures that there exists some constant $L=L(L_f,P)$ such that the following Lipschitz condition hold (cf. \cite{CuiH19}):
%\begin{align}
%&(F(u)-F(v), u-v)\leq L\|u-v\|^2,  \quad u, v\in L^{2P}, \notag\\
%& \|F(u)-F(v)\|\leq C(1+\|u\|^{P-1}_E+\|v\|_E^{P-1})\|u-v\|, \quad u,v \in E=C(\mathcal{O}).  
%\end{align}

\begin{assumption}\label{assump:3}
(Linear growth and Lipschitz condition for $g$) Given $\gamma>1$. The mapping $g(v)$ satisfies 
\begin{align*}
\|g(u)\|_{L_{2,\nu}^0}\leq c\|u\|_{\nu}, \ \ u\in H^{\nu}(\mathcal{O})
\end{align*}
with  $\nu=0$ and $\nu=\gamma$, and 
\begin{align*}
\|g(u)-g(v)\|\leq c\|u-v\|, \ u,v\in L^2(\mathcal{O}),
\end{align*}

\end{assumption}

\begin{rem} 
The paramter $\gamma$ essentially determines (see Theorem \ref{prop:H1r} below)  spatial regularity. It is clear that linear functions satisfy the assumption which  relaxes  the sublinear growth condition of $g$ to some extent (cf. \cite{AntoKM16}). 
% In particular, 
%if $g(u)=I$ in \eqref{model:AC}, and $A$ and $Q$ commute, this assumption leads to $\|(-A)^{\frac{\gamma-1+\epsilon}{2}}Q^{\frac{1}{2}}\|_{L_2}<\infty$, which  is the same assumption (up to an arbitrary small positive $\epsilon$)  prescribed in \cite{CuiH19, KovacsLL18, KovacsLL11, QiW19, Wang17} for miscellaneous SPDEs with additive noise. }

%In order to prove the strong convergence of our semi-discretization and full-discretization, we made a slightly stronger assumption on $g$ }
\end{rem}

\begin{assumption}\label{assump:4}
(Initial condition)  Let  $\gamma>1$ be the same as in {\bf Assumption \ref{assump:3}}. We assume that the initial condition $u_0$ is $\mathscr{F}_0/\mathscr{B}(H^{\gamma})$-measurable and 
$$
\mathbf{E}\|u_0\|_{\gamma}^p<\infty,\ \ p\geq 2.
$$
\end{assumption}

Under Assumptions \ref{assump:1}-\ref{assump:4} and $u_0\in C(\mathcal{O})$, there exists a unique predictable process $u$ (cf. \cite{Cerrai03, LiuR10}) such that for any $p\geq 1$, one has
\begin{align}\label{eq:uq}
\mathbf{E}\sup\limits_t\|u(t)\|^p<\infty. 
\end{align}
Based upon it, one further infers that 
\begin{align}\label{eq:fu}
\mathbf{E}\sup\limits_t\|f(u(t))\|<\infty. 
\end{align}

\begin{rem} (on the well-posedness of \eqref{model:AC})
\begin{itemize}
\item One may obtain the well-posedness of \eqref{model:AC} with \eqref{eq:uq} by virtue of the variational approach \cite[Chap. 5 and Appendix G]{LiuR05}. %We refer to \cite{CuiH19, MajeeP18} for convergence analysis of their numerical scheme by this approach. 
\item If both $f$ and $g$ are globally Lipschitz continuous with linear growth condition, then the well-posedness of \eqref{model:AC} is standard and has been provided in, for instance, \cite[Chap. 2]{Kruse14}.

\item If  $f(v)$ is a polynomial of degree $P$,
to guarantee the existence and uniqueness of solution for \eqref{model:AC} for cylindrical white noise (cf. \cite{Cerrai03}),  $g(u)$ is required to have the following restriction
$$
\|g(u)\|\leq C(1+\|u\|^{1/P}), \ u\in H.
$$
 
%\item A counterpart of Assumption \ref{assump:3} is $\|A^{\frac{\gamma-1}{2}}Q^{\frac{1}{2}}\|_{L_2(U)}<\infty$ provided the noise is addtive  and $A$ and $Q$ commute (cf. \cite{CuiH19, KovacsLL15,QiW19}); Our assumption is the same as it up to an arbitrarily small positive $\epsilon$ in this spirit;

\item The assumption 2.4 on initial condition is not essential since one may alleviate the assumption by exploring the smoothing effect of $E(t)$. \end{itemize}
\end{rem}

%\begin{rem}	This assumption appears because  $AP_N\neq P_NA$ in our case, which prohibits using the technique  in \eqref{eq:e3}. 	It is essentially different from Fourier spectral element discretization. 
%\end{rem}

\subsection{Spatial regularity of $u$}
We proceed to exploit the regularity of the solution \eqref{eq:truesol} under these assumptions. 
We note that an optimal spatial regularity has been established for additive noise under the conditions $\|(-A)^{\frac{\gamma-1}{2}}Q^{\frac{1}{2}}\|<\infty $ (cf. \cite{BrehierCH19,QiW19}). To simplify the notation, we shall omit the dependence on $x$ when no confusion can arise.

\begin{theorem}\label{prop:H1r}
 Under  Assumptions \ref{assump:1}-Assumption \ref{assump:4},  the unique mild solution $u(t)$ of \eqref{model:AC} satisfies 

\begin{align*} 
\mathbf{E}\sup\limits_t\|u(t)\|_\gamma^p,\; \mathbf{E}\sup\limits_t\|f(u(t))\|_{\gamma}^p <\infty\quad\forall p\geq 2.
\end{align*}
%{Furthermore, $\mathbf{E}\sup\limits_t\|f(u(t))\|_{\gamma}^p<\infty$. } 
\end{theorem}
\begin{proof}
We start with \eqref{eq:truesol}. For any $t>0$
\begin{align}\label{eq:ut1}
\|u(t)\|_{L^p(\Omega; H^{\gamma})}
&\leq \big\|(-A)^{\frac{\gamma}{2}}E(t)u_0\big\|_{L^p(\Omega;H)}+\bigg\|(-A)^{\frac{\gamma}{2}}\int_0^t E(t-\sigma)f(u(\sigma))d\sigma\bigg\|_{L^p(\Omega;H)}\notag\\
&\quad+\bigg\|(-A)^{\frac{\gamma}{2}}\int_0^t E(t-\sigma)g(u)dW^Q(\sigma)\bigg\|_{L^p(\Omega;H)}.
\end{align}
The assumption on $u_0: \Omega\to {H}^{\gamma}$ implies the bound for the first term
\begin{align}\label{eq:e1}
\big\|(-A)^{\frac{\gamma}{2}}E(t)u_0\big\|_{L^p(\Omega;H)}\leq \|u_0\|_{L^p(\Omega; {H}^{\gamma})}<C.
\end{align}
For the last term in \eqref{eq:ut1}, we use the Burkholder-Davis-Gundy inequality, Assumption \ref{assump:3} and generalized Gronwall inequality to obtain
\begin{align}\label{eq:e3}
&\bigg\|\int_{0}^t (-A)^{\frac{\gamma}{2}} E(t-\sigma)g(u(\sigma))dW^Q(\sigma)\bigg\|_{L^p(\Omega;H)}\notag\\
&\leq C(p)\bigg(\int_{0}^t \big(\mathbf{E}\big\|(-A)^{\frac{\gamma}{2}} E(t-\sigma)g(u(\sigma))\big\|_{L_2^0}^p\big)^{\frac{2}{p}}d\sigma\bigg)^{\frac{1}{2}}\notag\\
&\leq C(p)\bigg(\int_{0}^t \big(\mathbf{E}\big\|(-A)^{\frac{1-\epsilon}{2}} E(t-\sigma)(-A)^{\frac{\gamma-1+\epsilon}{2}}g(u(\sigma))\big\|_{L_2^0}^p\big)^{\frac{2}{p}}d\sigma\bigg)^{\frac{1}{2}}\notag\\%&\leq C\mathbf{E}\int_{0}^t  \| g(u(\sigma)) \|_{L_{2, \gamma}^0}^2d\sigma\notag\\
%&\leq C\mathbf{E}\int_{0}^t  \|g(u(\sigma))\|_{L_{2,\gamma}^0}^{2}d\sigma\notag\\
&\leq C\bigg(\int_{0}^t (t-\sigma)^{\epsilon-1} \big(\mathbf{E}\|g(u(\sigma))\|_{L^0_{2,\gamma-1+\epsilon}}^p\big)^{\frac{2}{p}}d\sigma\bigg)^{\frac{1}{2}}\notag\\
&\leq C\bigg(\int_{0}^t  (t-\sigma)^{\epsilon-1} \big(\mathbf{E}\|(u(\sigma))\|_{L^0_{2,\gamma-1+\epsilon}}^p\big)^{\frac{2}{p}}d\sigma\bigg)^{\frac{1}{2}}
\leq C\bigg(\int_{0}^t (t-\sigma)^{\epsilon-1}  \|u(\sigma)\|_{L^p(\Omega; H^\gamma)}^2d\sigma\bigg)^{\frac{1}{2}}
\end{align}

It remains to bound the second term in \eqref{eq:ut1}. Towards this end, we consider $\gamma$ in differently intervals separately as follows. 
 (i) Case  $\gamma\in (1,2)$:
\begin{align}\label{eq:e21}
\bigg\|(-A)^{\frac{\gamma}{2}}\int_0^t E(t-\sigma)f(u(\sigma))d\sigma\bigg\|_{L^p(\Omega;H)}
\leq \int_0^t (t-\sigma)^{-\frac{\gamma}{2}}\|f(u(\sigma))\|_{L^p(\Omega;H)}d\sigma
<C.
\end{align}
Therefore, $u(t)\in L^p(\Omega; H^\gamma)$ by the Gronwall's inequality using \eqref{eq:e1},\eqref{eq:e21} and \eqref{eq:e3}.  

Since $\gamma>1$, we have $H^\gamma(\mathcal{O})$ is a Banach algebra for $d=1,2$ (cf. \cite[Page 106]{Adams75}). Hence, 
$\mathbf{E}\sup\limits_t\|f(u(t))\|_{\gamma}^p<\infty$.

 (ii) Case $\gamma=2$:
 
 From the previous case, one has $u(t)\in L^p(\Omega; H^\mu  )$ for some $\mu\in (1,2)$. Hence,
\begin{align}\label{eq:e22}
&\bigg\|\int_0^t E(t-\sigma)Af(u(\sigma))d\sigma\bigg\|_{L^p(\Omega;H)}
=\bigg\|\int_0^t E(t-\sigma)(-A)^{1-\frac{\mu}{2}} (-A)^{\frac{\mu}{2}}f(u(\sigma))d\sigma\bigg\|_{L^p(\Omega;H)}\notag\\
&\leq \int_0^t (t-\sigma)^{\frac{\mu}{2}-1} \|f(u(\sigma))\|_{L^p(\Omega;H^\mu)}d\sigma<C.
\end{align}
Therefore, $u(t)\in L^p(\Omega; H^2)$ by the generalized Gronwall's inequality using \eqref{eq:e1},\eqref{eq:e22} and \eqref{eq:e3}, and by the same reason as in the previous case, 
$\mathbf{E}\sup\limits_t\|f(u(t))\|_{2}^p<\infty$.

  (iii) Case $\gamma\in (2,4)$:
   
 By virtue of the results of the previous case, 
  \begin{align}\label{eq:e2}
&\bigg\|\int_0^t (-A)^{\frac{\gamma}{2}}E(t-\sigma) f(u(\sigma))d\sigma\bigg\|_{L^p(\Omega;H)}
\leq \int_0^t (t-\sigma)^{1-\frac{\gamma}{2}}\|f(u(\sigma)\|_{L^p(\Omega;H^2)}d\sigma<C.
\end{align}
 
 We repeat the above process for arbitrarily large $\gamma$ as long as both \eqref{eq:e1} and \eqref{eq:e3} hold or Assumptions \ref{assump:3} and \ref{assump:4} hold. 

The proof is completed. 
\end{proof}
\begin{rem}
This theorem lifts an essential restriction on $\gamma$  in \cite{CuiH19, KovacsLL15, KovacsLL18, LarssonM11, QiW19}, and allows us to obtain higher-order convergence in space, as opposed to the low-order convergence rate of linear FEM approximation considered in \cite{CuiH19, KovacsLL15, KovacsLL18, LarssonM11, QiW19}. 
\end{rem}

%{\color{red} The next lemma is true for $d=3$ if we assume $\gamma> 3/2$, since $H^{\gamma}(\mathcal{O})\hookrightarrow L^\infty(\mathcal{O})$ if $\gamma> 3/2$ and $H^1(\mathcal{O})\hookrightarrow L^6(\mathcal{O})$. Please see if any subsequent results need to be modified!}

%{\color{blue} Reply: We only need to consider dimension in the proof of Theorem 2.1. Our proof is safe if we assume $\gamma>3/2$ when $d=3$. }

The next lemma establishes a local Lipschitz continuity for the nonlinear $f$. 
{
\begin{lemma}\label{lem:f}
Let $\gamma>1$. Then, under the assumption 2.2,  we have 
\begin{align}\label{eq:contf}
\|f(u)-f(v)\|\leq  C(1+\|u\|_{\gamma}^{P-1}+\|v\|_{\gamma}^{P-1})\|u-v\|, \ \ u,v\in H^{\gamma}(\mathcal{O}),
\end{align}
where $C$ is independent of $u$ and $v$.
%and 
%\begin{align}
%\|(-A)^{-\frac{1}{2}}(f(u)-f(v))\|\leq C\|u-v\|, \ \ \ u,v \in H^1(\mathcal{O}).
%\end{align}
\end{lemma}
\begin{proof}
Under the assumption on $\gamma$, we have $H^{\gamma}(\mathcal{O})\hookrightarrow L^\infty(\mathcal{O})$. Hence,
\begin{align*}
\|f(u)-f(v)\|&=\bigg\|\sum\limits_{j=0}^P a_j(u^j-v^j)\bigg\|\notag\\
&=\bigg\|(u-v)\sum\limits_{j=1}^Pa_j(u^{j-1}+u^{j-2}v+\cdots+v^{j-1})\bigg\|\notag\\
&\leq \|u-v\|\sum\limits_{j=1}^P |a_j| \bigg(\|u\|^{j-1}_{L^\infty}+\|u\|^{j-2}_{L^\infty}\|v\|_{L^\infty}+\cdots+\|v\|^{j-1}_{L^\infty}\bigg)\notag\\
&\leq C(1+\|u\|_{L^\infty}^{P-1}+\|v\|_{L^\infty}^{P-1})\|u-v\|\notag\\
&\leq C(1+\|u\|_{\gamma}^{P-1}+\|v\|_{\gamma}^{P-1})\|u-v\|.
%&\leq \|u-v\|\sum\limits_{j=1}^P |a_j| \bigg(\|u\|^{j-1}_{\gamma}+\|u\|^{j-2}_{\gamma}\|v\|_{\gamma}+\cdots+\|v\|^{j-1}_\gamma\bigg)
%\leq C\|u-v\|.
\end{align*}
%Here, $C$ depends on $u$ and $v$. It is bounded by a constant if $\gamma$ is sufficiently large due to the Sobolev embedding. 
%In addition, by embedding $H^1(\mathcal{O})\hookrightarrow L^q(\mathcal{O}), \ q\le 6$ and Holder's inequality
%\begin{align*}
%&\|(-A)^{-\frac{1}{2}}(f(u)-f(v))\|=\sup\limits_{w\in H}\frac{\langle (-A)^{-\frac{1}{2}}(f(u)-f(v)),w\rangle}{\|w\|}\notag\\
%&\quad=\sup\limits_{w\in H}\frac{\langle (f(u)-f(v)),(-A)^{-\frac{1}{2}}w\rangle}{\|w\|}=\sup\limits_{z\in H^1}\frac{\langle (f(u)-f(v)),z\rangle}{\|z\|_1}\notag\\
%&\quad\leq \sup\limits_{z\in H^1} \frac{\|f(u)-f(v)\|_{L^{6/5}}\|z\|_{L^6}}{\|z\|_1}\leq C\|f(u)-f(v)\|_{L^{6/5}}.
%%&\quad\leq C\bigg\|\int_0^1f'(su+(1-s)v)ds (u-v)\bigg\|_{L^{6/5}}\leq C\|u-v\| \int_0^1 \|f'(su+(1-s)v)\|_{L^3}ds\notag\\
%%&\quad\leq C\|u-v\|. 
%\end{align*}
\end{proof} }

%\begin{rem}
%In our applications of this lemma, the constants $C$ is uniform with respect to $t$, probability event $\omega$ though they may depend on $u$ and $v$, see Theorem \ref{prop:H1r}, Theorem \ref{lem:uNgam} and Theorem \ref{thm:uncond} below. Because $\gamma$ is large, 
%the nonlinear term $f(u)$ appears to own a globally Lipschitz continuity albeit it is a high-order polynomial with respect to $u$ (cf. \cite{KovacsLL11}).  
%\end{rem}

%The temporal regularity is important for our convergence of full-discretization.
%\begin{lemma}\label{lem:timereg}
%Let Assumption \ref{assump:1}-Assumption \ref{assump:4} be fulfilled. Then,
%\begin{align}
%\sup\limits_{s,t\in [0,T]} \frac{\mathbf{E}\|u(t)-u(s)\|^2}  {|t-s|^{\min\{1,\gamma\}}} \leq C.
%\end{align}
%\end{lemma}
%\begin{proof} Assume that $s<t$
%\begin{align} 
%\mathbf{E}\|u(t)-u(s)\|^2&=\mathbf{E}\bigg\|\int_{s}^t e^{(t-s)A}f(u(\theta))d\theta+\int_s^t e^{(t-s)A}g(u(\theta))dW^Q(\theta)\bigg\|^2\notag\\
%&\leq 2\mathbf{E}\bigg\|\int_{s}^t e^{(t-s)A}f(u(\theta))d\theta\bigg\|^2+2\mathbf{E}\bigg\|\int_s^t e^{(t-s)A}g(u(\theta))dW^Q(\theta)\bigg\|^2\notag\\
%\end{align}
%\end{proof}
\section{Spatial semi-discretization}
We describe below our spatial semi-discretization  and  carry out  an error analysis. We assume $\mathcal{O}=(0,1)^d,\;(d=1,2)$.
\subsection{Spatial semi-discretization}
 Let $\mathcal{P}_N$ be the space of polynomials on $\mathcal{O}$ with degree at most $N$ in each direction and  $V_N=\{v|v\in \mathcal{P}_N, v|_{\partial\mathcal{O}}=0\}$.
We define $P_N: H^{-1}\to V_N$  a generalized projection by (cf. \cite{Kruse14}):
\begin{align}
(P_N v, y_N)=(\nabla A^{-1}v, \nabla y_N), \ \forall v\in H^{-1},\ y_N\in  V_N.
\end{align}
It is clear that for $v\in L^2(\mathcal{O})$, we have 
\begin{align*}
	(P_N v, y_N)=(v,  y_N), \ \forall  y_N\in  V_N,
\end{align*}
from which we derive \cite{BerM97}
\begin{align}\label{errorPN}
	\|P_Nv-v\|\le \inf_{v_N\in V_N}\|v_N-v\|\le CN^{-r}\|u\|_r,\quad\forall r>0.
\end{align}
 We introduce a discrete operator $A_N: V_N\to V_N$ defined by
 $$
 \langle A_Nv_N, \chi_N\rangle:=-((-A)^{1/2} v_N, (-A)^{1/2} \chi_N),\ \ \forall v_N,\chi_N\in V_N.
 $$
% Clearly, $A_N$ is negative-definite. 
Then the spectral Galerkin approximation of \eqref{model:AC} yields
\begin{align}\label{eq:semidis}
du_N=A_Nu_Ndt+P_Nf(u_N)dt+P_Ng(u_N)dW^Q(t),\;\; u_N(0)=P_Nu_0.
\end{align}
Similar as the continuous case, there exists a unique mild solution $u_N$ to \eqref{eq:semidis} which can be written as
\begin{align}\label{eq:uN}
u_N(t)=E_N(t)P_Nu_0+\int_0^t E_N(t-s)P_Nf(u_N(s))ds+\int_0^t E_N(t-s)P_N\big[g(u_N)dW^Q(s)\big],
\end{align}
where $E_N(t)=e^{tA_N}$. 
Similar to \cite[Lemma 3.9]{Thomee97}, one has the property
\begin{align}\label{eq:ANEN}
\|(-A_N)^\mu E_N v_N\|\leq Ct^{-\mu}\|v_N\| \ \ \text{for\ all}\ t>0, \ v_N\in \mathcal{P}_N,
\end{align}
and  defines  the operator
\begin{align*}
F_N(t):=E_N(t)P_N-E(t).
\end{align*}

\begin{lemma} \label{lem:1}
 Let $0\leq \nu\leq \mu$. Then there exists a constant $C$ such that 
\begin{align*}
\|F_N(t)u\|\leq CN^{-\mu}t^{-\frac{\mu-\nu}{2}}\|u\|_{\nu}, \ \forall u\in {H}^\nu.
\end{align*}
% (ii) Let $0\leq\rho\leq 1$. Then there exists a constant $C$ such that
% \begin{align*}
% \|F_N(t)x\|\leq CN^{\rho-2}t^{-1}\|x\|_{-\rho}, \ \forall x\in {H}^{-\rho}.
% \end{align*}
% (ii) Let $0\leq\rho\leq 1$. There exists a constant $C$ such that
% \begin{align*}
% \bigg(\int_0^t \|F_N(s)x\|^2ds\bigg)^{1/2}\leq CN^{-1-\rho}\int_0^t\|E(s)x\|_{1+\rho}ds,\ \forall x\in {H}^\rho.
% \end{align*}
\end{lemma}
\begin{proof}
	Thanks to \eqref{errorPN}, this result can be proved by using the same technique used for finite elements (cf. \cite[Theorem 3.5]{Thomee97}), so we omit the detail here. 
\end{proof}

%\begin{comment}
\begin{lemma}\label{lem:uNgam} 
	
Let Assumptions \ref{assump:1}-\ref{assump:4} hold and $u_N$ is given by \eqref{eq:semidis}. Then, for all $p\geq 2$, 
\begin{align*} 
\mathbf{E}\sup\limits_t\|u_N(t)\|^p<C,
\end{align*}
where $C$ is independent of $N$. 
\end{lemma}
\begin{proof} %From the proof of Theorem \ref{prop:H1r}, we only need to prove $\mathbf{E}\sup\limits_t\|u_N(t)\|^p<\infty$. 
This can be done by following the arguments in \cite{LiuR10} as follows.

By \eqref{eq:semidis}, Ito's formula and Assumption \ref{assump:2}, we have
\begin{align}
\|u_N(t)\|^p&=\|P_Nu_0\|^p+p(p-2)\int_0^t \|u_N(s)\|^{q-4}\|(P_Ng(u_N(s)P_N))^\ast u_N(s)\|^2ds\notag\\
&\quad+p\int_0^t \|u_N(s)\|^{p-2}\langle A_Nu_N(s)+P_Nf(u_N(s)), u_N(s)\rangle +\frac{1}{2}\|P_Ng(u_N(s))P_N\|_{L_2(H)}^2ds\notag\\
&\quad+p\int_0^t \|u_N(s)\|^{p-2} \langle u_N(s), P_Ng(u_N(s))dW^Q(s)\rangle\notag\\
&\leq \|P_Nu_0\|^p+C\int_0^t \|u_N(s)\|^pds-p\int_0^t \|u_N(s)\|^{p-2} \|\nabla u_N(s)\|^2ds-\theta\int_0^t \|u_N(s)\|^{p-2+\alpha}ds\notag\\
&\quad+q\int_0^t \|u_N(s)\|^{p-2} \langle u_N(s), P_Ng(u_N(s))dW^Q(s)\rangle. 
\end{align}

For any given $N$, we define the stopping time
$$
\tau_R^N=\min\{\inf\{t\in [0,T]: \|u_N(t)\|>R\}, T\}, \; R>0. 
$$
It is obvious that 
$$
\lim\limits_{R\to\infty} \tau_R^N=T, \;\; \mathbf{P}\text{-}a.s.$$

Then by the Burkholder-Davis-Gundy inequality and the Young's inequality, we have
\begin{align}
&\mathbf{E}\sup\limits_{r\in [0, t]} \bigg|\int_0^r\|u_N(s)\|^{p-2}\langle u_N(s), P_Ng(u_N(s))dW^Q(s)\bigg|\notag\\
&\quad\leq 3\mathbf{E}\bigg(\int_0^t \|u_N(s)\|^{2p-2}\|g(u_N(s))\|_{L_2^0}^2ds\bigg)^{1/2}\notag\\
&\quad\leq 3\mathbf{E}\bigg(\epsilon\sup\limits_{s\in [0,t]}\|u_N(s)\|^{2p-2}\cdot C_\epsilon\int_0^t \|(u_N(s))\|^2ds\bigg)^{1/2}\notag\\
&\quad\leq 3\mathbf{E}\bigg[\epsilon\sup\limits_{s\in [0,t]}\|u_N(s)\|^{p}+C_\epsilon\bigg(\int_0^t \|(u_N(s))\|^2ds\bigg)^{p/2}\bigg]\notag\\
&\quad\leq 3\epsilon\mathbf{E}\sup\limits_{s\in [0,t]}\|u_N(s)\|^{p}+C_\epsilon\int_0^t \mathbf{E}\|u_N(s)\|^pds.
\end{align}
Therefore, the Gronwall's inequality implies
\begin{align*}
\mathbf{E}\sup\limits_{t\in [0, \tau_R^N]} \|u_N(t)\|^p\leq C\mathbf{E}\|u_0\|^p,\;\;n\geq 1. 
\end{align*}
For $R\to\infty,$ the desired result follows from the monotone convergence theorem. 
\end{proof}
%\end{comment}
%Because of Theorem \ref{prop:H1r}, Theorem \ref{lem:uNgam} and Theorem \ref{thm:uncond} below, there exists a set $\Omega_\epsilon\subset\Omega$ with $\mathbf{P}(\Omega_\epsilon)>1-\epsilon$ and such that  (cf. \cite{KovacsLL11})
%\begin{align}
%\|u(t)\|_\gamma^p+\|u_N(t)\|_\gamma^p\leq C/\epsilon.
%\end{align}

\begin{theorem}\label{prop:semi}
Let $u$ and $u_N$ be the solutions of \eqref{model:AC} and \eqref{eq:weakform}. Then, under Assumptions \ref{assump:1}-\ref{assump:4},  there exists a constant $C$ independent of $N$ such that 
\begin{align}
\|u(t)-u_N(t)\|_{L^2(\Omega; H)}\leq CN^{-\gamma}, \ \ t>0.
\end{align} 
\end{theorem}
\begin{proof}Let $e^N(t)=u(t)-u_N(t)$. 
Subtracting \eqref{eq:uN} from \eqref{eq:truesol} and multiplying  the result  by $e_N(t)$ gives
\begin{align*}
\|e^N(t)\|^2&= (F_N(t)u_0, e_N(t))+\int_0^t (E_N(t-s)P_Nf(u_N(s))-E(t-s)f(u(s)), e_N(t))ds\notag\\
&\quad+\bigg(\int_0^t E_N(t-s)P_N[g(u_N(s))dW^Q(s)-E(t-s)g(u(s))dW^Q(s)], e_N(t)\bigg)\notag\\
%&\quad+\bigg\|\int_0^t E_N(t-s)P_Ng(u_N(s)) dW^Q(s)-E_N(t-s)P_N[g(u_N(s))dW_J^Q(s)]\bigg\|_{L^2(\Omega;U)}\notag\\
&\quad:=I_1+I_2+I_3.
\end{align*}

The first term can be estimated by Lemma \ref{lem:1} with $\mu=\nu=\gamma$:
\begin{align}\label{eq:semiI1}
\mathbf{E}I_1\leq\mathbf{E}\|F_N(t)u_0\|^2+\frac{1}{4}\mathbf{E}\|e_N(t)\|^2\leq CN^{-2\gamma}\|u_0\|_{\gamma}^2+\frac{1}{4}\mathbf{E}\|e_N(t)\|^2.
\end{align}
The second one can be separated by two terms as follows
\begin{align*}
&\int_0^t (E_N(t-s)P_Nf(u_N(s))-E(t-s)f(u(s)), e_N(t))ds\notag\\
&\quad\leq\int_0^t (F_N(t-s)f(u(s)), e_N(t))ds
 +\int_0^t \bigg(E_N(t-s)P_N(f(u(s))-f(u_N(s))), e_N(t)\bigg)ds\notag\\
&\quad =:I_{21}+I_{22}.
\end{align*}
{An application of Young's inequality, together with Theorem \ref{prop:H1r}, 
%one can  infer that $\mathbf{E}\sup\|f(u(t))\|_\gamma<\infty$.  
and Lemma \ref{lem:1} (with $\mu=\nu=\gamma$) gives
\begin{align*}
\mathbf{E}I_{21}\leq CN^{-2\gamma}\mathbf{E}\sup\limits_t\|f(u(t))\|_{\gamma}^2+\frac{1}{4}\mathbf{E}\|e_N(t)\|^2. 
\end{align*}
}
%Similarly, 
%\begin{align*}
%I_{22}\leq CN^{-\gamma}\int_0^t \|f(u(t))-f(u(s))\|_{L^2(\Omega; H^{\gamma})}ds\leq CN^{-\gamma}\|f(u(t))\|_{L^2(\Omega;H^{\gamma})}\leq CN^{-\gamma}.
%\end{align*}
%Since $L^2\hookrightarrow H^{-(1-r)}$,  \eqref{eq:contf} implies
%\begin{align*}
%\|f(u(t))-f(u(s))&\|_{L^p(\Omega;H^{r-1})}\leq \|f(u(t))-f(u(s))\|_{L^p(\Omega;H)}\notag\\
%&\leq C\|u(t)-u(s)\|_{L^p(\Omega; H^{1+r})}\leq C|t-s|^{1/2}.
%\end{align*}
%Therefore, $I_{22}\leq CN^{-1-r}$ as desired. 
In order to bound $I_{22}$, we apply  the one-sided Lipschitz condition for $f$, thanks to  Theorem \ref{prop:H1r} and Lemma \ref{lem:uNgam} , we   obtain
\begin{align*}
\mathbf{E}I_{22}\leq L\int_0^t  \mathbf{E}\|e_N(s)\|^2ds.
\end{align*}
Therefore, a combination of estimations of $I_{21}$ and $I_{22}$ yields
\begin{align}\label{eq:I2}
\mathbf{E}I_2\leq L\int_0^t  \mathbf{E}\|e_N(s)\|^2ds+CN^{-2\gamma}+\frac{1}{4}\mathbf{E}\|e_N(t)\|^2.
\end{align}

Similarly, the Young's inequality implies
%\begin{align*}
%I_3&= \bigg(\mathbf{E}\bigg[\int_0^t\|E_N(t-s)P_Ng(u_N(s))-E(t-s)g(u(s))\|_{L_2^0}^2ds\bigg]\bigg)^{\frac{1}{2}}\\
%&\leq C\bigg\|\bigg(\int_0^t\|E_N(t-s)P_N(g(u_N(s))-g(u(s)))\|_{L_2^0}^2ds\bigg)^{\frac{1}{2}}   \bigg\|_{L^2(\Omega; R)}\\
%&\quad+ C\bigg\|\bigg(\int_0^t\|F_N(t-s)(g(u(s))-g(u(t)))\|_{L_2^0}^2ds\bigg)^{\frac{1}{2}}   \bigg\|_{L^2(\Omega; R)}\\
%&\quad+C\bigg\|\bigg(\int_0^t\|F_N(t-s)g(u(t))\|_{L_2^0}^2ds\bigg)^{\frac{1}{2}}   \bigg\|_{L^2(\Omega; R)}\\
%&=:I_{31}+I_{32}+I_{33},
%\end{align*}
%where $L_2^0=L^2(U_0;H)$ and $U_0=Q^{1/2}(U)$ with $U=H_0^1(\mathcal{O})$. 
\begin{align*}
\mathbf{E}I_3&
\leq \mathbf{E}\bigg\|\int_0^t E_N(t-s)P_N[g(u_N(s))dW^Q(s)-E(t-s)g(u(s))dW^Q(s)]\bigg\|^2+\frac{1}{4}\mathbf{E}\|e_N(t)\|^2\notag\\
&\leq \frac{1}{4}\mathbf{E}\|e_N(t)\|^2+2\mathbf{E}\bigg\|\int_0^t F_N(t-s)g(u(s))dW^Q(s)\bigg\|^2\notag\\
&\quad+2\mathbf{E}\bigg\|\int_0^t E_N(t-s)P_N(g(u(s))-g(u_N(s))dW^Q(s)\bigg\|^2\notag\\
&:=\frac{1}{4}\mathbf{E}\|e_N(t)\|^2+I_{31}+I_{32}.
\end{align*}
The Burkholder-Davis-Gundy inequality, Lemma \ref{lem:1},  Assumption \ref{assump:3} imply
\begin{align}
I_{31}&\leq 2C^2(p)\int_0^t \mathbf{E}\|F_N(t-s)g(u(s))\|_{L_2^0}^2ds\notag\\
%&\leq C^2(p)N^{-2\gamma}\int_0^t \big(\|A^{\frac{1-\epsilon}{2}}E(t-s)\|_{L(H)}^p\mathbf{E}\|g(u)\|_{L_{2,\gamma-1+\epsilon}}^p\big)^{\frac{2}{p}}ds\notag\\
&\leq CN^{-2\gamma}\int_0^t \mathbf{E}\|g(u)\|_{L_{2,\gamma}}^2ds
\leq CN^{-2\gamma}\int_0^t \mathbf{E}\|u(s)\|^2_{\gamma}ds.
\end{align}
Following the same spirit, we have
\begin{align}
I_{32}&\leq 2C^2(p)\int_0^t \mathbf{E}\|E_N(t-s)P_N(g(u(s))-g(u_N(s))\|^2_{L_2^0}ds\notag\\
&\leq 2C^2(p)\int_0^t \mathbf{E}\|g(u(s))-g(u_N(s)\|^2ds\notag\\
&\leq C\int_0^t \mathbf{E}\|u(s)-u_N(s)\|^2ds.
%&\leq C\int_0^t \|u(s)-u_N(s)\|^2_{L^p(\Omega;H)}ds.
\end{align}

Hence,
\begin{align}\label{eq:I3}
I_3\leq CN^{-2\gamma}+C\int_0^t\mathbf{E} \|e_N(s)\|^2ds+\frac{1}{4}\mathbf{E}\|e_N(t)\|^2. 
\end{align}

Finally, we combine estimates \eqref{eq:semiI1}-\eqref{eq:I3} and arrive at
\begin{align*}
\mathbf{E}\|e_N(t)\|^2&\leq CN^{-2\gamma}+C\int_0^t \mathbf{E}\|e_N(s)\|^2ds
%&\quad+C\bigg(\int_0^t\|u_N(s)-u(s)\|_{L^p(\Omega;H)}^2ds\bigg)^{1/2}.
%&\leq CN^{-1-r}+C\bigg(\int_0^t(t-s)^{-1/2}\|u(s)-u_N(s)\|_{L_2(\Omega;H)}^2ds\bigg)^{1/2}
\end{align*}
%Then, by Cauchy-Schwartz inequality
%\begin{align*}
%\|u(t)-u_N(t)\|^2_{L^p(\Omega;H)}&\leq CN^{-2\gamma}+C\bigg(\int_0^t \|u(s)-u_N(s)\|_{L^p(\Omega;H)}ds\bigg)^2\notag\\
%&\quad+C\int_0^t\|u_N(s)-u(s)\|_{L^p(\Omega;H)}^2ds, \notag\\
%&\leq CN^{-2\gamma}+C\int_0^t \|u(s)-u_N(s)\|_{L^p(\Omega;H)}^2ds.
%\end{align*}
Then, the desired result is achieved by the  Gronwall's inequality. 
\end{proof}
\subsection{Efficient implementation with spectral-Galerkin method}
We present below an efficient implementation by using the spectral-Galerkin method  \cite{Shen94} which will greatly simply the implementation and increase the efficiency. 
 To fix the idea, we  take $\mathcal{O}=(0,1)^2$ and $A=\Delta$ as an example. 

Our spectral semi-discretization  \eqref{eq:semidis} is to equivalent to finding $u_N\in V_N$ such that
\begin{align}\label{eq:weakform}
(d u_N, \chi_N)=(A_Nu_N,\chi_N)dt+(f(u_N),\chi_N)dt+(g(u_N)dW^Q(t),\chi_N),\ \ \chi_N\in V_N,
\end{align}
where $W^Q(t)\approx \sum\limits_{j_1,j_2=1}^J \sqrt{q_{j_1j_2}}e_{j_1j_2}(x,y)\beta_{j_1j_2}(t)$. 

Let $\{\phi_{m}(\cdot)\}_{m=1}^N$ be the basis functions of $V_N$ in 1-D so that 
 $\{\phi_{m}(x)\phi_j(y)\}_{m,j=1}^N$  forms a basic for $V_N$ in 2-D.
\begin{align*} 
&u_N(t)=\sum\limits_{m,n=0}^{N-2} c_{mn}(t)\phi_{m}(x)\phi_n(y),\;   \mathbf{C}(t)=(c_{mn}(t))_{m,n=0,1,\cdots, N-2}; \notag\\
&a_{mn}=\int_0^1\phi'_m(x)\phi'_n(x)dx,\; \mathbf{A}=(a_{mn})_{m,n=0,1,\cdots, N-2};\notag\\
&b_{mn}=\int_0^1\phi_m(x)\phi_n(x)dx,\; \mathbf{B}=(b_{mn})_{m,n=0,1,\cdots, N-2};\notag\\
&f_{mn}=\int_{\mathcal{O}}f(u_N(t))\phi_m(x)\phi_n(y)dxdy,\; \mathbf{F}(t)=(f_{mn})_{m,n=0,1,\cdots, N-2};\notag\\
&g^{j_1j_2}_{mn}=\int_{\mathcal{O}}g(u_N(t))e_{j_1j_2}(x,y)\phi_m(x)\phi_n(y)dxdy,\; \mathbf{G}_{j_1j_2}(t)=(g^{j_1j_2}_{m,n})_{m,n=0,1,\cdots,N-2}.
\end{align*}
Then, \eqref{eq:weakform} can be transformed into
\begin{align}\label{eq:C}
\mathbf{B}(d{\mathbf{C}}(t))\mathbf{B}=-[\mathbf{A}\mathbf{C}(t)\mathbf{B}+\mathbf{B}\mathbf{C}(t)\mathbf{A}]dt+\mathbf{F}(t)dt+\sum\limits_{j_1,j_2=1}^J \sqrt{q_{j_1j_2}}\mathbf{G}_{j_1j_2}(t)d\beta_{j_1j_2}(t).
\end{align}
We now perform a matrix diagonalization technique (cf. \cite[Chap 8]{ShenTW11}) to the above system. Let $(\lambda_i,\bar h_i)$ $(i=0,1,\cdots,N-2)$ be the generalized eigenpairs such that $\mathbf{B}\bar h_i=\lambda_i\mathbf{A}\bar h_i $, and set 
\begin{equation}\label{eigendecomp}
 \mathbf{\Lambda}=\text{diag}(\lambda_0,\lambda_1,\cdots,\lambda_{N-2}),\;
 \mathbf{H}=(\bar h_0,\bar h_1,\cdots,\bar h_{N-2}).
\end{equation}
 Then, we have $\mathbf{B}\mathbf{H}=\mathbf{A}\mathbf{H\Lambda}$. Note that since $\mathbf{A}$ and $\mathbf{B}$ are symmetric, we have $\mathbf{H}^{-1}=\mathbf{H}^T$.
 
Writing $\mathbf{C}(t)=\mathbf{H}\mathbf{V}(t)\mathbf{H}^T$ in \eqref{eq:C}, we arrive at
\begin{align*}
%\begin{cases}
\mathbf{H}\mathbf{\Lambda} d\mathbf{V}(t) \mathbf{\Lambda}\mathbf{H}^T=-[\mathbf{HV}(t)\mathbf{\Lambda H}^T+\mathbf{H\Lambda V}(t)\mathbf{H}^T]dt+\mathbf{A}^{-1}(\mathbf{F}(t)dt+\sum\limits_{j_1,j_2=1}^J \sqrt{q_{j_1j_2}}\mathbf{G}_{j_1j_2}(t)d\beta_{j_1j_2}(t))\mathbf{A}^{-1}.
%\mathbf{V}(0)=\mathbf{H}^T\mathbf{U}_0\mathbf{H}.
%\end{cases}
\end{align*}
Multiplying the left (resp. right) of the above equation by $\mathbf{H}^T$ (resp. $\mathbf{H}$), we arrive at
\begin{align*}
%\begin{cases}
\mathbf{\Lambda} d\mathbf{V}(t) \mathbf{\Lambda}=-[\mathbf{V}(t)\mathbf{\Lambda }+\mathbf{\Lambda V}(t)]dt+\mathbf{H}^T\mathbf{A}^{-1}(\mathbf{F}(t)dt+\sum\limits_{j_1,j_2=1}^J \sqrt{q_{j_1j_2}}\mathbf{G}_{j_1j_2}(t)d\beta_{j_1j_2}(t))\mathbf{A}^{-1}\mathbf{H},
%\mathbf{V}(0)=\mathbf{H}^T\mathbf{U}_0\mathbf{H}.
%\end{cases}
\end{align*}
which can be rewritten componentwise as a system of nonlinear SDEs with decoupled linear parts:
\begin{align}
\lambda_m\lambda_n d{V}_{mn}(t)&=-[\lambda_m+\lambda_n]{V}_{mn}(t) dt+(\mathbf{H}^T\mathbf{A}^{-1}\mathbf{F}(t)\mathbf{A}^{-1}\mathbf{H})_{mn}dt\notag\\
&\quad+\sum\limits_{j_1,j_2=1}^J \sqrt{q_{j_1j_2}}(\mathbf{H}^T\mathbf{A}^{-1}\mathbf{G}_{j_1j_2}(t)\mathbf{A}^{-1}\mathbf{H})_{mn}d\beta_{j_1j_2}(t),\;\;
 0\leq m,n\leq N-2. \label{decoup}
\end{align}

Several remarks are in order:
\begin{itemize}
\item 
In principle, one can solve the above system of nonlinear SDEs using any standard SDE solver. We shall construct a special tamed semi-implicit scheme in the next section which is unconditionally stable as well as extremely easy to implement.
\item The above procedure is also applicable to a separable operator $A$ in the form $Au=\partial_x (a(x) \partial_x u)+\partial_y (b(y) \partial_y u)$, and can be extended in a straightforward fashion to three dimensions.
\item In the special case of $A=\Delta$ considered above, we can use
\begin{align}\label{eq:phim}
 \phi_{m}(x)=\frac{1}{2\sqrt{4m+6}}(L_m(x)-L_{m+2}(x)),\ \ m\geq 0,\end{align}
 where  $L_m(x)$ is the shifted Legendre polynomials on $[0,1]$ such that
 $\phi_m(0)=\phi_m(1)=0$ and $(\phi'_m,\phi'_n)=\delta_{mn}$ \cite{Shen94}. 
Hence, $\mathbf{A}$ is the identity matrix, and the entries of $\mathbf{B}$ has the explicit form \cite{Shen94}
\begin{align}\label{matrixB}
b_{mn}=b_{nm}=
\begin{cases}
\displaystyle{\frac{1}{4(4m+6)}\big(\frac{1}{2m+1}+\frac{1}{2m+5}\big),\ \ m=n}\\
\displaystyle{-\frac{1}{4\sqrt{(4m+6)(4m-2)}}\frac{1}{2m+1},\  \;  m=n+2}\\
\displaystyle{0}, \quad\qquad\qquad\qquad\qquad\qquad \;\;\;\;\;\quad\text{otherwise}.
\end{cases}
\end{align}
Therefore, the eigenpairs of $\mathbf{B}$ can be efficiently and accurately computed.
\end{itemize}

\section{Full discretization and its error analysis}
In this section, we present our  full discretized scheme, establish its stability and carry out    its convergence analysis. 
\subsection{A tamed semi-implicit scheme}
Let $\tau$ be the time step size and $M=T/\tau$. We start with a first-order semi-discrete tamed time discretization scheme for \eqref{model:AC}:
\begin{align}\label{eq:umc}
 &u^{k+1}-u^{k}=\tau\Delta u^{k+1}+\frac{\tau f(u^k)}{1+\tau\|f(u^k)\|^2}+  g(u^k)\Delta W^Q(t_k),
\;0\leq k\leq M-1. 
\end{align}
Combining with \eqref{eq:weakform}, we have its fully discretized version:
\begin{align}\label{eq:um}
 &(u_N^{k+1}-u_N^{k}, \psi)=\tau(\Delta u_N^{k+1}, \psi)+\frac{\tau (f(u_N^k), \psi)}{1+\tau\|f(u_N^k)\|^2}+ (g(u_N^k)\Delta W^Q(t_k),\psi),\ \ \psi\in V_N,\notag\\
 \intertext{or}
 &u_N^{k+1}-u_N^k=\tau A_Nu_N^{k+1}+\frac{\tau P_Nf(u_N^k)}{1+\tau\|f(u_N^k)\|^2}+ P_N[g(u_N^k)\Delta W^Q(t_k)],
\;0\leq k\leq M-1. 
\end{align}
A remarkable property of the semi-discrete  tamed  schemes is that, despite treating the nonlinear term and noise term explicitly, there are still unconditionally stable as we show below.
\begin{theorem}\label{thm:uncond}
The schemes  \eqref{eq:umc} and  \eqref{eq:um} admit a unique solution $u^{k+1}$ and $u_N^{k+1}$, and are unconditionally stable in the sense that for $1\le k\le M-1$, we have
\begin{align*}
\mathbf{E}\| u^{k+1}\|^2+2\tau\sum\limits_{j=0}^k \mathbf{E}\|\nabla u^{j+1}\|^2\leq \exp((2K+2c^2Tr(Q))T)\mathbf{E}\|u_0\|^2+\frac{ \exp((2K+2c^2Tr(Q))T)}{K+c^2Tr(Q)},
\end{align*}
and
\begin{align*}
\mathbf{E}\| u_N^{k+1}\|^2+2\tau\sum\limits_{j=0}^k \mathbf{E}\|\nabla u_N^{j+1}\|^2\leq \exp((2K+2c^2Tr(Q))T)\mathbf{E}\|u_N^0\|^2+\frac{ \exp((2K+2c^2Tr(Q))T)}{K+c^2Tr(Q)},
\end{align*}
where $c$ and $K$ are constants from our Assumptions \ref{assump:1}-\ref{assump:4}.
\end{theorem}
\begin{proof}

The proof for the semi-discrete and full-discrete cases are essentially the same so we shall only prove the result for the full-discrete case.

It is clear that the scheme  \eqref{eq:um}  admits a unique solution. 

Choosing $\psi=u_N^{k+1}$ in \eqref{eq:um} and using the identity
$$
(a-b,2a)=|a|^2-|b|^2+|a-b|^2,
$$
we obtain
\begin{align*}
&\frac{1}{2}[\| u_N^{k+1}\|^2-\| u_N^k\|^2+\|(u_N^{k+1}-u_N^k)\|^2]+\tau\|\nabla u_N^{k+1}\|^2\notag\\
&=\frac{\tau}{1+\tau\|f(u_N^k)\|^2}(f(u_N^k), u_N^{k+1})+(g(u_N^k)\Delta W^Q(t_k),u_N^{k+1}).
\end{align*}
Then,  the Young's inequality, Assumption \ref{assump:2} and \ref{assump:3},  and {$\mathbf{E}[(g(u_N^k)\Delta W^Q(t_k), u_N^k)\big| \mathcal{F}_{t_k}]=0$}  imply
\begin{align*}
&\frac{1}{2}\mathbf{E}[\|u_N^{k+1}\|^2-\| u_N^k\|^2+\|(u_N^{k+1}-u_N^k)\|^2]+\tau\mathbf{E}\|\nabla u_N^{k+1}\|^2\notag\\
&\quad\leq \mathbf{E}\bigg[\frac{\tau}{1+\tau\|f(u_N^k)\|^2}(f(u_N^k), u_N^{k+1}-u_N^k)\bigg)\bigg]+\tau\mathbf{E} (f(u_N^k), u_N^k)
+\mathbf{E}(g(u_N^k)\Delta W^Q(t_k),u_N^{k+1}-u_N^k)\notag\\
&\quad\leq \mathbf{E}\bigg[\frac{\tau}{1+\tau\|f(u_N^k)\|^2}\bigg(\tau \|f(u_N^k)\|^2+\frac{1}{4\tau}\|u_N^{k+1}-u_N^k\|^2\bigg)\bigg]+K\tau\mathbf{E}\|u_N^k\|^2\notag\\
&\quad\quad +\mathbf{E}\|g(u_N^k)\Delta W^Q(t_k)\|^2+\frac{\mathbf{E}\|u_N^{k+1}-u_N^k\|^2}{4}\notag\\
&\quad\leq \tau+\frac{1}{2}\mathbf{E}\|u_N^{k+1}-u_N^k\|^2+K\tau\mathbf{E}\|u_N^k\|^2+c^2\tau Tr(Q)\mathbf{E}\|u_N^k\|^2, %?? \text{ more detail about the last term is needed??}
\end{align*}
{ where we have used $\Delta W^Q(t_k)\sim N(0, Q\tau)$ in the last line.}

So, we have
\begin{align*}
\mathbf{E}\| u_N^{k+1}\|^2-\big(1+(2K+2c^2Tr(Q))\tau\big)\mathbf{E}\|u_N^k\|^2+2\tau\mathbf{E}\|\nabla u_N^{k+1}\|^2\leq 2\tau.
%&\leq 2\tau+[1+(2K+2c^2Tr(Q))\tau] \bigg(2\tau+[1+(2K+2c^2Tr(Q))\tau]\mathbf{E}\|u_N^{k-1}\|^2\bigg)\notag\\
%&\leq \cdots\leq [1+(2K+2c^2Tr(Q))\tau]^{k+1}\mathbf{E}\|u_0\|^2+\sum\limits_{j=0}^k [1+(2K+2c^2Tr(Q))\tau]^j 2\tau.
\end{align*}
Denote 
$$A_0=\big(1+(2K+2c^2Tr(Q))\tau\big),\ \ \text{and}\ \  k+1=M\gamma_0, \ 0\leq \gamma_0\leq 1.$$
We have 
\begin{align*}
&\ \mathbf{E}\| u_N^{k+1}\|^2-A_0\mathbf{E}\|u_N^k\|^2+2\tau\mathbf{E}\|\nabla u_N^{k+1}\|^2\leq 2\tau,\notag\\
&A_0\mathbf{E}\| u_N^{k}\|^2-A_0^2\mathbf{E}\|u_N^{k-1}\|^2+2A_0\tau\mathbf{E}\|\nabla u_N^{k}\|^2\leq 2A_0\tau,\notag\\
&\cdots\notag\\
&A_0^k\mathbf{E}\| u_N^{1}\|^2-A_0^{k+1}\mathbf{E}\|u_N^0\|^2+2A_0^{k}\tau\mathbf{E}\|\nabla u_N^{1}\|^2\leq 2A_0^k\tau.
\end{align*}
Summing up the above inequalities yields
\begin{align*}
\mathbf{E}\| u_N^{k+1}\|^2-A_0^{k+1}\mathbf{E}\|u_N^0\|^2+2\tau\sum\limits_{j=0}^k A_0^{k-j}\mathbf{E}\|\nabla u_N^{j+1}\|^2\leq \frac{2\tau(1-A_0^{k+1})}{1-A_0}.
\end{align*}
 Moreover,  a simple computation shows that
$$
\lim\limits_{M\to\infty}A_0^{k+1}=\lim\limits_{M\to\infty} \big(1+(2K+2c^2Tr(Q))\tau\big)^{k+1}=\exp((2K+2c^2Tr(Q))T\gamma_0).
$$
Therefore, we obtain
\begin{align*}
\mathbf{E}\| u_N^{k+1}\|^2+2\tau\sum\limits_{j=0}^k A_0^{k-j}\mathbf{E}\|\nabla u_N^{j+1}\|^2\leq \exp((2K+2c^2Tr(Q))T)\mathbf{E}\|u_N^0\|^2+\frac{ \exp((2K+2c^2Tr(Q))T)}{K+c^2Tr(Q)}.
\end{align*}
The desired result follows since $A_0>1$. 
\end{proof}
\subsection{Convergence analysis}

Now, we  carry out a convergence analysis for \eqref{eq:um}. We denote
$$
E^{n}=(I-\tau A_N)^{-n}, 1\leq n\leq M,
$$ 
%which  has the following smoothing property \cite[Lemma 7.3]{Thomee97}:
%\begin{align}\label{eq:smooth_En}
%\|A^jE^nv\|\leq Ct_n^{-j}\|v\|, \ \ t_n\geq t_j,\ j\geq 0, \ \ v\in H,
%\end{align}
which has the following approximation properties. 
\begin{lemma} Under Assumption \ref{assump:1}, we have
\begin{align}\label{eq:estimate_En}
&\|(-A)^{\frac{\rho}{2}}(E(t_n)-E^n)v\|\leq C\tau^{\frac{\beta}{2}} t_n^{-\frac{\beta-\gamma+\rho}{2}}\|(-A)^{\frac{\gamma}{2}}v\|, \ 0\leq\gamma\leq \beta+\rho, \gamma\geq 0, \beta\in [0,2];\notag\\
& \|(E(t)-E^nP_N)v\|\leq C(N^{-\mu}+\tau^{\min\{\frac{\mu}{2},1\}}) \|v\|_\mu,\ \ v\in H^\mu. 
\end{align}
\end{lemma}
\begin{proof}
The first inequality can be found in \cite{ KovacsLL15}.  We only need to prove the second one. 
\begin{align*}
 \|(E(t)-E^nP_N)v\|\leq \|(E^nP_N-E(t_n))v\|+\|(E(t_n)-E(t))v\|, \; t\in [t_{n-1},t_n].
\end{align*}
It is clear that
\begin{align*}
\|(E(t_n)-E(t))v\|&=\|(-A)^{-\frac{\mu}{2}}(E(t_n-t))-I)E(t)(-A)^{\frac{\mu}{2}}v\|\\
&\leq C(t_n-t)^{\frac{\mu}{2}}\|x\|_\mu\leq C\tau^{\frac{\mu}{2}}\|v\|_\mu,
\end{align*}
where \eqref{eq:ineq2} is applied and therefore we require $0\leq \mu<2$ for this estimate.  

Furthermore, since $v$ is smooth, we can follow the proof of \cite[Theorem 7.8]{Thomee97} to derive
\begin{align*}
\|(E^nP_N-E(t_n))v\|\leq C(N^{-\mu}\|v\|_\mu+\tau\|v\|_2),\ \; t_n\geq 0.
\end{align*}
 Note that for this estimate, we only require $v\in \dot{H}^\mu$, where $\mu$ can be arbitrarily large. 
\end{proof}

\begin{rem}
 From the proof of \eqref{eq:estimate_En}, one easily infer that the spatial error can be made arbitrarily small provided $v$ is sufficiently smooth whereas the temporal error is at most of order $\mathcal{O}(\tau)$ which can not be improved. 
\end{rem}
We start by establishing  some temporal properties of $u(s)$. 
{\color{black}
\begin{lemma}\label{lem:contu}
Under  Assumptions \ref{assump:1}-\ref{assump:4}, we have
\begin{align}
\|u(t)-u(s)\|_{L^p(\Omega;H)}\leq C(t-s)^{\min\{\frac{\gamma}{2}, \frac{1}{2}\}}, \ \ p\geq 2.
\end{align} 
\end{lemma}
\begin{proof}
Suppose that $0\leq s\leq t\leq T$. Using \eqref{eq:truesol},
\begin{align}
&\|u(t)-u(s)\|_{L^p(\Omega;H)}\notag\\
&\leq \|(E(t)-E(s))u_0\|_{L^p(\Omega;H)}+\bigg\|\int_0^t E(t-\sigma)f(u(\sigma))d\sigma-\int_0^sE(s-\sigma)f(u(\sigma))d\sigma\bigg\|_{L^p(\Omega;H)}\notag\\
&\quad+\bigg\|\int_0^t E(t-\sigma)g(u(\sigma))dW^Q(\sigma)-\int_0^sE(s-\sigma)g(u(\sigma))dW^Q(\sigma)\bigg\|_{L^p(\Omega;H)}\notag\\
&:=H_1+H_2+H_3.
\end{align}
Using \eqref{eq:ineq2},
\begin{align}
H_1&\leq \|E(s)(-A)^{-\min\{\frac{\gamma}{2},1\}}(E(t-s)-I)(-A)^{\min\{\frac{\gamma}{2},1\}}u_0\|_{L^p(\Omega;H)}\notag\\
&\leq C(t-s)^{\min\{\frac{\gamma}{2},1\}}\|u_0\|_{L^p(\Omega;H^{\min\{\gamma,2\}})}.
\end{align}
Similarly,

\begin{align}
H_2&\leq \bigg\|\int_0^s \big(E(t-\sigma)-E(s-\sigma)\big)f(u(\sigma)))d\sigma\bigg\|_{L^p(\Omega;H)}+\bigg\|\int_s^t E(t-\sigma)f(u(\sigma))d\sigma\bigg\|_{L^p(\Omega;H)}\notag\\
&\leq C(t-s)^{{\min\{\frac{\gamma}{2},1\}}} \int_0^s \|f(u(\sigma))\|_{L^p(\Omega;H^{\min\{\gamma,2\}})}d\sigma+\int_s^t \|E(t-\sigma)f(u(\sigma))\|_{L^p(\Omega;H)}ds\notag\\
&\leq C(t-s)^{\min\{\frac{\gamma}{2},1\}}. 
\end{align}

By the Burkholder-Davis-Gundy inequality and Theorem \ref{prop:H1r}
\begin{align}
H_3^2&\leq C\mathbf{E}\bigg\|\int_0^s \big(E(t-\sigma)-E(s-\sigma)\big)g(u(\sigma))dW^Q(\sigma)\bigg\|_{L^{p}(\Omega;H)}^2\notag\\
&\quad+C\mathbf{E}\bigg\|\int_s^tE(t-\sigma)g(u(\sigma))dW^Q(\sigma)\bigg\|_{L^p(\Omega;H)}^2\notag\\
&\leq C\int_0^s (\mathbf{E}\|\big(E(t-\sigma)-E(s-\sigma)\big)g(u(\sigma))\|^p_{L_2^0})^{\frac{2}{p}}d\sigma+C\int_s^t(\mathbf{E}\|E(t-\sigma)g(u(\sigma))\|^p_{L_2^0})^{\frac{2}{p}}d\sigma\notag\\
&\leq C\int_0^s (\mathbf{E}\|(-A)^{\frac{1-\epsilon}{2}} E(s-\sigma) (-A)^{-\min\{\frac{\gamma}{2},1\}}(E(t-s)-I) (-A)^{\min\{\frac{\gamma}{2},1\}-\frac{1-\epsilon}{2}} g(u(\sigma))\|^p_{L_2^0})^{\frac{2}{p}}d\sigma\notag\\
&\quad+C(t-s)\notag\\
&\leq C(t-s)^{\min\{\gamma,2\}}\int_0^s (s-\sigma)^{\epsilon-1} (\mathbf{E}\|g(u)\|^p_{\gamma-1+\epsilon})^{\frac{2}{p}}d\sigma+C(t-s)\notag\\
&\leq C(t-s)^{\min\{\gamma,2\}} (\mathbf{E}\sup_\sigma\|u(\sigma)\|_{\gamma-1+\epsilon}^p)^{\frac{2}{p}}+C(t-s).
\end{align}
The result follows by combing estimates of $H_1$, $H_2$ and $H_3$.
\end{proof}
}

\begin{theorem}\label{prop:full}
  Let
$u(t)$ and $u_N^m$ be  solutions of \eqref{eq:truesol} and \eqref{eq:um} respectively. Then,  under  Assumptions \ref{assump:1}-\ref{assump:4}, there exists a constant $C$ independent of $N$ and $\tau$ such that
\begin{align}
\|u(t_m)-u_N^m\|_{L_2(\Omega; H)}\leq C(N^{-\gamma}+\tau^{\min\{\frac{\gamma}{2}, \frac{1}{2}\}}), \ \ t>0.
\end{align} 
\end{theorem}
\begin{proof} 
Following the idea from \cite{QiW20}, we introduce an auxiliary process
\begin{align}
\tilde{u}_N^{n}-\tilde{u}_N^{n-1}=\tau A_N\tilde{u}^{n}+\frac{\tau P_Nf(u(t_{n-1}))}{1+\tau\|f(u(t_{n}))\|^2}+P_Ng(u(t_{n-1}))\Delta W^Q(t_{n}),
\end{align}
which can be rewritten as
\begin{align}\label{eq:tildeu}
\tilde{u}_N^{n}=E^{n}P_Nu_0+\tau\sum\limits_{k=1}^{n} \frac{E^{n-k}P_Nf(u(t_{k-1}))}{1+\tau\|f(u(t_{k-1}))\|^2}+\sum\limits_{k=1}^n \int_{t_{k-1}}^{t_k}E^{n-k}P_Ng(u(t_{k-1}))dW^Q(s).
\end{align}
By the proof of Theorem \ref{prop:H1r}, it is straightforward to infer that $\mathbf{E}\|\tilde{u}_N^n\|_\gamma^p<\infty$. Consequently, $\mathbf{E}\|f(\tilde{u}_N^{n})\|^2<\infty$, for all $1\leq n\leq M$.

Note that \eqref{eq:um} can also be written in closed form
\begin{align}\label{eq:un2}
u_N^n=E^nP_Nu_0+\sum\limits_{k=1}^{n}\int_{t_{k-1}}^{t_k}\frac{E^{n-k}P_Nf(u_N^{k-1})}{1+\tau\|f(u_N^{k-1})\|^2}ds
+\sum\limits_{k=1}^{n}\int_{t_{k-1}}^{t_k}E^{n-k}P_Ng(u_N^{k-1})dW^Q(s).
\end{align}

Next, we split the error $\|u(t_n)-u_N^n\|_{L^2(\Omega; H)}, 1\leq n\leq M$ into two parts, and bound them individually.
\begin{align}
\|u(t_n)-u_N^n\|_{L^2(\Omega; H)}\leq \|u(t_n)-\tilde{u}_N^n\|_{L^2(\Omega; H)}+\|\tilde{u}_N^n-u_N^n\|_{L^2(\Omega; H)}. 
\end{align}
Subtracting \eqref{eq:tildeu} from \eqref{eq:truesol} and taking the associated norm gives
\begin{align}
&\|u(t_n)-\tilde{u}_N^n\|_{L^p(\Omega; H)}\notag\\
&\quad\leq \|(E(t_n)-E^nP_N)u_0\|_{L^p(\Omega;H)}+\bigg\|\int_0^{t_n}E(t_n-s)f(u(s))ds-\tau\sum\limits_{k=1}^{n} \frac{E^{n-k}P_Nf(u(t_{k-1}))}{1+\tau\|f(u(t_{k-1}))\|^2}\bigg\|_{L^p(\Omega;H)}\notag\\
&\qquad+\bigg\|\int_0^{t_n}E(t_n-s)g(u(s))dW^Q(s)-\sum\limits_{k=1}^n \int_{t_{k-1}}^{t_k}E^{n-k}P_Ng(u_N^{k-1})dW^Q(s)\bigg\|_{L^p(\Omega;H)}\notag\\
&\qquad:=I_1+I_2+I_3.
\end{align}
An application of \eqref{eq:estimate_En} gives
\begin{align}
I_1\leq C(N^{-\gamma}+\tau^{\min\{\frac{\gamma}{2},1\}})\|u_0\|_{L^p(\Omega;H^\gamma)}.
\end{align}
$I_2$ can be decomposed in the following way:
\begin{align}
I_2&=\bigg\|\sum\limits_{k=1}^n\int_{t_{k-1}}^{t_k}\bigg[E(t_n-s)f(u(s))-\frac{E^{n-k}P_Nf(u(t_{k-1}))}{1+\tau\|f(u(t_{k-1}))\|^2}\bigg]ds\bigg\|_{L^p(\Omega;H)}\notag\\
&\leq \bigg\|\sum\limits_{k=1}^n\int_{t_{k-1}}^{t_k}E(t_n-s)[f(u(s))-f(u(t_{k-1}))]ds\bigg\|_{L^p(\Omega;H)}\notag\\
&\quad+\bigg\|\sum\limits_{k=1}^n\int_{t_{k-1}}^{t_k}[E(t_n-s)-E^{n-k}]f(u(t_{k-1}))ds\bigg\|_{L^p(\Omega;H)}\notag\\
&\quad+\bigg\|\sum\limits_{k=1}^n\int_{t_{k-1}}^{t_k}\bigg[E^{n-k}f(u(t_{k-1}))-\frac{E^{n-k}P_Nf(u(t_{k-1}))}{1+\tau\|f(u(t_{k-1}))\|^2}\bigg]ds\bigg\|_{L^p(\Omega;H)}\notag\\
&\quad:=I_{21}+I_{22}+I_{23}.
\end{align}

Since $\gamma$ is sufficiently large, we can directly resort to Lemma \ref{lem:f} and Lemma \ref{lem:contu} to bound $I_{21}$. Otherwise, one has to follow  \cite{JentzenKW11, QiW19} to bound this term. 
\begin{align*}
 \|I_{21}\|_{L^p(\Omega:H)}&=\bigg\|\sum\limits_{k=1}^n\int_{t_{k-1}}^{t_k} E(t_n-s)(f(u(s))-f(u(t_{k-1})))ds\bigg\|_{L^p(\Omega;H)}\notag\\
 &\leq C\sum\limits_{k=1}^n \int_{t_{k-1}}^{t_k} \big(1+\|u(s)\|_{L^{2p}(\Omega; H^\gamma)}^{2P-2}+\|u(t_{k-1})\|_{L^{2p}(\Omega; H^\gamma)}^{2P-2}\big) \|u(s)-u(t_{k-1})\|_{L^{2p}(\Omega;H)}\notag\\
&\leq C\sum\limits_{k=1}^n\int_{t_{k-1}}^{t_k} \|u(s)-u(t_{k-1})\|_{L^{2p}(\Omega;H)}ds\notag\\
&\leq C\tau^{\min\{\frac{\gamma}{2},\frac{1}{2}\}}.
\end{align*}
By \eqref{eq:estimate_En}, 
\begin{align}
\|I_{22}\|_{L^p(\Omega;H)}&\leq \bigg\|\sum\limits_{k=1}^n\int_{t_{k-1}}^{t_k}[E(t_n-s)-E^{n-k}P_N]f(u(t_{k-1}))ds\bigg\|_{L^p(\Omega;H)}\notag\\
&\quad+\bigg\|\sum\limits_{k=1}^n\int_{t_{k-1}}^{t_k}[E^{n-k}P_N-E^{n-k}]f(u(t_{k-1}))ds\bigg\|_{L^p(\Omega;H)}\notag\\
&\leq C(N^{-\gamma}+\tau^{\min\{\frac{\gamma}{2},1\}}).
\end{align}
Similarly, using Theorem \ref{prop:H1r} and \eqref{eq:estimate_En} gives
\begin{align}
&\|I_{23}\|_{L^p(\Omega;H)}= \bigg\|\sum\limits_{k=1}^n\int_{t_{k-1}}^{t_k}\bigg[E^{n-k}f(u(t_{k-1}))-\frac{E^{n-k}P_Nf(u(t_{k-1}))}{1+\tau\|f(u(t_{k-1}))\|^2}\bigg]ds\bigg\|_{L^p(\Omega;H)}\notag\\
&\leq \sum\limits_{k=1}^n\int_{t_{k-1}}^{t_k} \bigg\|\frac{(E^{n-k}-E^{n-k}P_N)f(u(t_{k-1}))
+\tau\|f(u(t_{k-1}))\|^2 E^{n-k}f(u(t_{k-1}))}{1+\tau\|f(u(t_{k-1}))\|^2}\bigg\|_{L^p(\Omega;H)}ds\notag\\
&\leq C(N^{-\gamma}+\tau^{\min\{\frac{\gamma}{2},1\}}\|f(u)\|_{L^p(\Omega;H^\gamma)}+C\tau(\|f(u(t_{k-1}))\|_{L^{4p}(\Omega;H)}^4+\|f(u(t_{k-1}))\|^2_{L^{2p}(\Omega;H)})\notag\\
&\leq C(N^{-\gamma}+\tau^{\min\{\frac{\gamma}{2},1\}}).
\end{align}
Hence, 
\begin{align}
\|I_2\|_{L^p(\Omega;H)}\leq C(N^{-\gamma}+\tau^{\min\{\frac{\gamma}{2},\frac{1}{2}\}}).
\end{align}
$I_3$ can be bounded by  using the Burkholder-Davis-Gundy inequality, Assumption \ref{assump:3},  Lemma \ref{lem:contu}, and \eqref{eq:estimate_En}. Note that $\|Q^{\frac{1}{2}}\|_{L_2}<\infty$. 
\begin{align}
\|I_3\|_{L^p(\Omega;H)}^2&=\bigg\|\sum\limits_{k=1}^n \int_{t_{k-1}}^{t_k}[E(t_n-s)g(u(s))-E^{n-k}P_Ng(u(t_{k-1})]dW^Q(s) \bigg\|^2_{L^p(\Omega;H)}\notag\\
%&\leq C\bigg(\int_0^{t_n} \bigg\|E(t_n-s)g(u(s)-\sum\limits_{k=1}^n1_{s\in [t_{k-1},t_k)}E^{n-k}P_Ng(u(t_{k-1})\bigg\|_{L^p(\Omega;H)}^2ds\bigg)^{\frac{1}{2}}\notag\\
&\leq C\sum\limits_{k=1}^n\int_{t_{k-1}}^{t_k}\|E(t_n-s)g(u(s))-E^{n-k}P_Ng(u(t_{k-1})\|^{2}_{L^p(\Omega;L_2^0)}ds\notag\\
&\leq C\sum\limits_{k=1}^n\int_{t_{k-1}}^{t_k}\|E(t_n-s)(g(u(s))-g(u(t_{k-1})))\|^{2}_{L^p(\Omega;L_2)}ds\notag\\
&\quad+C\sum\limits_{k=1}^n\int_{t_{k-1}}^{t_k}\|(E(t_n-s)-E^{n-k}P_N)g(u(t_{k-1}))\|^{2}_{L^p(\Omega;L_2)}ds\notag\\
&\leq C\sum\limits_{k=1}^n\int_{t_{k-1}}^{t_k}\|u(s)-u(t_{k-1})\|^{2}_{L^p(\Omega;L_2)}ds
+C(N^{-2\gamma}+\tau^{\min\{\gamma,2\}})\notag\\
&\leq C(N^{-2\gamma}+\tau^{\min\{\gamma,1\}}).
\end{align}
Thus, we can obtain
\begin{align}\label{eq:tildee}
&\|u(t_n)-\tilde{u}_N^n\|_{L^p(\Omega; H)}\leq C(N^{-\gamma}+\tau^{\min\{\frac{\gamma}{2},\frac{1}{2}\}}). 
\end{align}

Next, we estimate $\|\tilde{u}_N^n-u_N^n\|_{L^p(\Omega;H)}$. Denote $\tilde{e}_n:=\tilde{u}_N^n-u_N^n$. It is clear that $\tilde{e}_n$ satisfies the equation
\begin{align}
&\tilde{e}^n-\tilde{e}^{n-1}=\tau A_N\tilde{e}^n+\frac{\tau P_Nf(u(t_{n-1})}{1+\tau\|f(u(t_{n-1}))\|^2}-\frac{\tau f(u_N^{n-1})}{1+\tau\|f(u_N^{n-1})\|^2}\quad\\
&\qquad\qquad\qquad+P_N(g(u(t_{n-1}))-g(u_N^{n-1}))\Delta W^Q(t_n), \notag\\
&\tilde{e}_0=0. 
\end{align}
Multiplying both sides by $\tilde{e}^n$ gives
\begin{align}\label{eq:main-e}
\frac{1}{2}\|\tilde{e}^n\|^2&-\frac{1}{2}\|\tilde{e}^{n-1}\|^2+\frac{1}{2}\|\tilde{e}^n-\tilde{e}^{n-1}\|^2+\tau \|\nabla \tilde{e}^n\|^2\notag\\
&=\bigg\langle \frac{\tau f(u(t_{n-1})}{1+\tau\|f(u(t_{n-1}))\|^2}-\frac{\tau f(u_N^{n-1})}{1+\tau\|f(u_N^{n-1})\|^2}, \tilde{e}^n \bigg\rangle\notag\\
& \quad+ \langle (g(u(t_{n-1}))-g(u_N^{n-1}))\Delta W^Q(t_n), \tilde{e}^n \rangle\notag\\
&:=J+K.
\end{align}
A careful computation implies
\begin{align}
J&=\bigg|\bigg\langle \frac{\tau (f(u(t_{n-1})-f(u_N^{n-1}))+\tau^2(\|f(u_N^{n-1})\|^2f(u(t_{n-1})-\|f(u(t_{n-1}))\|^2 f(u_N^{n-1}))}{(1+\tau\|f(u(t_{n-1}))\|^2)(1+\tau\|f(u_N^{n-1})\|^2)}, \tilde{e}^n \bigg\rangle\bigg|\notag\\
&< \tau |\langle f(u(t_{n-1})-f(u_N^{n-1}),  \tilde{e}^n\rangle |+\tau^2\langle  \|f(u_N^{n-1})\|^2f(u(t_{n-1}))-\|f(u(t_{n-1}))\|^2f(u_N^{n-1}),\tilde{e}^n\rangle\notag\\
&:=J_1+J_2.
\end{align}
The estimate of $J_2$ is simple and we first bound it.
\begin{align}
J_2&\leq \tau^2\bigg(\frac{\|\tilde{e}^n\|^2}{2}+\frac{1}{2}\bigg\|\|f(u_N^{n-1})\|^2f(u(t_{n-1}))-\|f(u(t_{n-1}))\|^2f(u_N^{n-1})  \bigg\|^2\bigg)\notag\\
 &\leq \tau^2\bigg(\frac{\|\tilde{e}^n\|^2}{2}+\|f(u_N^{n-1})\|^4 \|f(u(t_{n-1})\|^2+\|f(u(t_{n-1}))\|^4\|f(u_N^{n-1})\|^2\bigg).
\end{align}
Then, we derive
\begin{align}
\mathbf{E}J_2&\leq \frac{\tau^2}{2}\mathbf{E}\|\tilde{e}^n\|^2+\frac{\tau^2}{2}(\mathbf{E}\|f(u_N^{n-1})\|^8+\mathbf{E}\|f(u(t_{n-1})\|^4+\mathbf{E}\|f(u(t_{n-1})\|^8+\mathbf{E}\|f(u_N^{n-1})\|^4).
%&\leq \frac{\tau^2}{2}\mathbf{E}\|\tilde{e}^n\|^2+C\tau,
\end{align}
%where we have used Theorem \ref{prop:H1r} and Theorem \ref{thm:uncond}  and the embedding $H^1\hookrightarrow L^p$.
%\begin{align}\label{eq:bduN}
%\tau \|f(u_N^{n-1})\|^p\leq C\tau \|\nabla u_N^{n-1}\|\leq C, \ \ \ p>1.
%\end{align}
Next, let us turn to $J_1$. We apply the one-sided Lipschitz condition for $f$, 
\begin{align}
J_1&\leq \tau|\langle f(u(t_{n-1})-f(\tilde{u}^{n-1}),  \tilde{e}^n\rangle |+\tau |\langle f(\tilde{u}^{n-1})-f(u_N^{n-1}),  \tilde{e}^{n}\rangle | \notag\\
&\leq \tau|\langle f(u(t_{n-1})-f(\tilde{u}^{n-1}),  \tilde{e}^n\rangle |+\tau |\langle f(\tilde{u}^{n-1})-f(u_N^{n-1}),  \tilde{e}^{n-1}\rangle |\notag\\
&\quad+\tau |\langle f(\tilde{u}^{n-1})-f(u_N^{n-1}),  \tilde{e}^{n}- \tilde{e}^{n-1}\rangle |\notag\\
&\leq \tau\bigg[\|f(u(t_{n-1})-f(\tilde{u}^{n-1})\|^2+\frac{1}{4}\|\tilde{e}^n\|^2\bigg]+L\tau\|\tilde{e}^{n-1}\|^2\notag\\
&\quad+\tau^2\|f(\tilde{u}^{n-1})-f(u_N^{n-1}) \|^2+\frac{1}{4}\|\tilde{e}^{n}- \tilde{e}^{n-1} \|^2.
\end{align}
Since  $\mathbf{E}\|\tilde{u}^{n-1})\|_\gamma^2<\infty$ and $\mathbf{E}\|u(t_n)\|_\gamma^2<\infty$, an application of Lemma \ref{lem:f} gives
\begin{align}
\mathbf{E}J_1&\leq C\tau \mathbf{E}\|\tilde{u}^{n-1}-u(t_{n-1})\|^2+\frac{\tau}{4}\mathbf{E}\|\tilde{e}^n\|^2+L\tau\mathbf{E}\|\tilde{e}^{n-1}\|^2\notag\\
&\quad +C\tau^2\mathbf{E}(\|f(\tilde{u}^{n-1})\|^2+\|f(u_N^{n-1}) \|^2)+
\frac{1}{4}\mathbf{E}\|\tilde{e}^{n}- \tilde{e}^{n-1} \|^2.
%&\leq C\tau (N^{-2\gamma}+\tau^{\min\{\gamma, 1\}})+\frac{\tau}{4}\mathbf{E}\|\tilde{e}^{n}\|^2+L\tau\mathbf{E}\|\tilde{e}^{n-1}\|^2+C\tau^2\mathbf{E}(\|f(\tilde{u}^{n-1})\|^2+\|f(u_N^{n-1}) \|^2)+
%\frac{1}{4}\mathbf{E}\|\tilde{e}^{n}- \tilde{e}^{n-1} \|^2.
\end{align}
%where we used \eqref{eq:tildee} and \eqref{eq:bduN}. 
Therefore, 
\begin{align}\label{eq:J}
\mathbf{E}J&\leq C\tau (N^{-2\gamma}+\tau^{\min\{\gamma, 1\}})+\bigg(\frac{\tau}{4}+\frac{\tau^2}{2}\bigg)\mathbf{E}\|\tilde{e}^n\|^2+L\tau\mathbf{E}\|\tilde{e}^{n-1}\|^2+
\frac{1}{4}\mathbf{E}\|\tilde{e}^{n}- \tilde{e}^{n-1} \|^2\notag\\
&\quad+C\tau^2[\mathbf{E}\|f(u_N^{n-1})\|^8
+\mathbf{E}\|f(u(t_{n-1})\|^4+\mathbf{E}\|f(u(t_{n-1})\|^8+\mathbf{E}\|f(u_N^{n-1})\|^4\notag\\
&\quad+\mathbf{E}\|f(\tilde{u}^{n-1})\|^2+\mathbf{E}\|f(u_N^{n-1}) \|^2]
\end{align}
Now it remains to bound $K$. By Assumption \ref{assump:3} and \eqref{eq:tildee},
\begin{align}\label{eq:K}
\mathbf{E}K&= \mathbf{E}\langle (g(u(t_{n-1}))-g(u_N^{n-1}))\Delta W^Q(t_n), \tilde{e}^n-\tilde{e}^{n-1} \rangle\notag\\
&\leq \mathbf{E} \|g(u(t_{n-1}))-g(u_N^{n-1}))\Delta W^Q(t_n)\|^2+\frac{1}{4}\mathbf{E} \|\tilde{e}^n-\tilde{e}^{n-1}\|^2\notag\\
&\leq C\tau Tr(Q) \mathbf{E}\|g(u(t_{n-1}))-g(u_N^{n-1}))\|^2+\frac{1}{4}\mathbf{E} \|\tilde{e}^n-\tilde{e}^{n-1}\|^2\notag\\
&\leq C\tau \mathbf{E}\|u(t_{n-1})-u_N^{n-1}\|^2+\frac{1}{4}\mathbf{E} \|\tilde{e}^n-\tilde{e}^{n-1}\|^2\notag\\
&\leq C\tau \mathbf{E}\|u(t_{n-1})-\tilde{u}_N^{n-1}\|^2+C\tau\mathbf{E}\|\tilde{e}^{n-1}\|^2 +\frac{1}{4}\mathbf{E} \|\tilde{e}^n-\tilde{e}^{n-1}\|^2
\notag\\
&\leq C\tau (N^{-2\gamma}+\tau^{\min\{{\gamma},1\}})+C\tau\mathbf{E}\|\tilde{e}^{n-1}\|^2 +\frac{1}{4}\mathbf{E} \|\tilde{e}^n-\tilde{e}^{n-1}\|^2.
\end{align}
Hence, substituting \eqref{eq:J}  and \eqref{eq:K} into \eqref{eq:main-e} and taking expectation,  we have for $\tau$ sufficiently small
\begin{align}
\mathbf{E}\|\tilde{e}^n\|^2\leq A(\tau)\mathbf{E}\|\tilde{e}^{n-1}\|^2+C\tau(N^{-2\gamma}+\tau^{\min\{{\gamma},1\}})+C\tau^2 B_{n-1}, 
\end{align}
where 
\begin{align*}
&A(\tau)=\frac{1+L\tau}{1-\frac{\tau}{4}-C\tau^2}, \quad
B_{n-1}:=\mathbf{E}\|f(u_N^{n-1})\|^8
+\mathbf{E}\|f(u(t_{n-1})\|^4+\mathbf{E}\|f(u(t_{n-1})\|^8\notag\\
&\qquad\qquad\qquad\qquad\qquad\qquad\qquad+\mathbf{E}\|f(u_N^{n-1})\|^4+\mathbf{E}\|f(\tilde{u}^{n-1})\|^2+\mathbf{E}\|f(u_N^{n-1}) \|^2.
\end{align*}
By a simple calculation,
\begin{align*}
\lim\limits_{n\to\infty} A(\tau)^n=e^{(L+\frac{1}{4})T},
\end{align*}
and by Theorem \ref{prop:H1r} and Theorem \ref{thm:uncond} and the embedding $H^1\hookrightarrow L^p, p\geq 2$
\begin{align*}
C\tau^2\sum\limits_{k=1}^{n-1} B_{k-1}\leq C\tau^2\sum\limits_{k=1}^{n-1} (\|\nabla u(t_k)\|^2+\|\nabla u_N^k\|^2)\leq C\tau.
\end{align*}
Therefore, 
\begin{align}\label{eq:etild1}
\mathbf{E}\|\tilde{e}^n\|^2&\leq C\tau^2\sum\limits_{k=1}^{n-1} A(\tau)^{n-k}B_{k-1}+ C\tau (N^{-2\gamma}+\tau^{\min\{{\gamma},1\}})\sum\limits_{k=1}^{n-1} A(\tau)^k\notag\\
&\leq C(N^{-2\gamma}+\tau^{\min\{{\gamma},1\}})+C\tau.
\end{align}
The result follows by a combination of \eqref{eq:tildee} and \eqref{eq:etild1}. 
\end{proof}

\subsection{Efficient implementation with spectral-Galerkin method}
We now describe our algorithm for implementing \eqref{eq:um} using the spectral-Galerkin described in the last section. To fix the idea, we consider $d=2$ and $A=\Delta$. In this case, we have $\mathbf{A}=I$ and  $\mathbf{B}$ is given by \eqref{matrixB}.
%Writing $\mathcal{F}_{t_{k}}$ adapted $H$-valued random variables $(u_N^{k})_{k=1}^M$ as 
Writing
$
u_N^{k}=\sum\limits_{m,n=0}^{N-2} c_{mn}^{k}\phi_{m}(x)\phi_n(y)$ in \eqref{eq:um},  setting $\mathbf{C}^{k}=(c_{mn}^{k})=\mathbf{HV}^{k}\mathbf{H}^T$ with 
 $\mathbf{V}^k=(V_{mn}^k)_{m,n=0,1,\cdots, N-2}$ and $\mathbf{H}$ is the eigenmatrix of $\mathbf{B}$ defined in \eqref{eigendecomp},
 and recalling \eqref{decoup}, we derive
\begin{align}\label{eq:V}
{\lambda_m\lambda_n}\frac{{V}_{mn}^{k}-{V}_{mn}^{k-1} }{\tau}&+ {(\lambda_m+\lambda_n)}{V}_{mn}^{k}= \bigg(\frac{1}{1+\tau\|f(u_N^{k-1})\|^2}\bigg)(\mathbf{H}^T\mathbf{F}^{k-1}\mathbf{H})_{mn}\notag\\
&\quad+\sum\limits_{j_1,j_2=1}^J 
{\sqrt{q_{j_1j_2}}}(\mathbf{H}^T\mathbf{G}_{j_1j_2}^{k-1} \mathbf{H})_{mn}  \Delta \beta_{j_1j_2}(t_{k-1}).
\end{align}
Here $\Delta\beta_{j_1j_2}(t_{k-1})$ are i.i.d random variables following $N(0,\tau)$-distribution and 
\begin{align}
%&u_N^{k}=\sum\limits_{m,n=0}^{N-2} c_{mn}^{k}\phi_{m}(x)\phi_n(y), \mathbf{C}^{k}=(c_{mn}^{k})=\mathbf{HV}^{k}\mathbf{H}^T, \mathbf{V}^k=(V_{mn}^k)_{m,n=0,1,\cdots, N-2};\notag\\
&\mathbf{F}^{k-1}=(f_{mn}^{k-1}), \;\;f_{mn}^{k-1}=\int_{\mathcal{O}}f(u_N^{k-1})\phi_m(x)\phi_n(y)dxdy;\notag\\
&\mathbf{G}_{j_1j_2}^{k-1}=(g^{j_1j_2,k-1}_{mn}),\;\; g^{j_1j_2,k-1}_{mn}=\int_{\mathcal{O}}g(u_N^{k-1})e_{j_1j_2}(x,y)\phi_m(x)\phi_n(y)dxdy.
\end{align}
Hence, we can determine ${V}_{mn}^{k}$ explicitly from \eqref{eq:V}.

Note that in general $\mathbf{F}^{k-1}$ and  $\mathbf{G}_{j_1j_2}^{k-1}$ can not be computed exactly. In practice, the following pseudo-spectral approach is used to approximately compute  $\mathbf{F}^{k-1}$ and  $\mathbf{G}_{j_1j_2}^{k-1}$.
Let $\{x_i,y_i\}_{i=0,N}$ be the Legendre-Gauss lobatto points, and $\mathbf{P}_N$ be the set of polynomials with degree less or equal than $N$ in each direction. We define an interpolation operator $I_N: C(\bar{\mathcal{O}})\rightarrow P_N$ such that $I_Nu(x_i,y_j)=u(x_i,y_j),\;i,j=01,1,\cdots,N$. Then, we approximate  $\mathbf{F}^{k-1}$ and  $\mathbf{G}_{j_1j_2}^{k-1}$ as follows:
\begin{align}
&f_{mn}^{k-1}\approx\int_{\mathcal{O}}I_N(f(u_N^{k-1}))\phi_m(x)\phi_n(y)dxdy;\notag\\
& g^{j_1j_2,k-1}_{mn}\approx\int_{\mathcal{O}} I_N(g(u_N^{k-1})e_{j_1j_2}(x,y))\phi_m(x)\phi_n(y)dxdy.
\end{align}
Since  $I_N(f(u_N^{k-1}))\in \mathbf{P}_N$, we can determine $h^{k-1}_{mn}$ such that
$I_N(f(u_N^{k-1}))=\sum_{m,n=0}^N h^{k-1}_{mn} L_n(x)L_m(y)$ where $\{L_j(\cdot)\}$ are the shifted Legendre polynomials. Hence, $f_{mn}^{k-1}$ can be easily obtained using the orthogonality of Legendre polynomials. The total cost of computing $\mathbf{H}^T\mathbf{F}^{k-1}\mathbf{H}$ is $O(N^{d+1})$ for the $d$-dimensional problem. 
One can compute 
$g^{j_1j_2,k-1}_{mn}$ in a similar way with the total cost of computing $\mathbf{H}^T\mathbf{G}_{j_1j_2}^{k-1} \mathbf{H}$ is $J^dN^{d+1}$ for the $d$-dimensional problem.

In summary, our algorithm can be described as follows:
\begin{enumerate}
\item Compute the eigenvalues and eigenvectors of the generalized eigenvalue problem $\mathbf{BH}=\mathbf{H\Lambda}$;% and compute $\mathbf{H}^{-1}$ (if $A$ is not symmetric);
\item Find $\mathbf{C}^0$  by projecting $u_0$ onto $\mathcal{P}_N\otimes \mathcal{P}_N$;
\item At time step $t_{k-1}$, compute $\mathbf{F}^{k-1}$, $\mathbf{G}_{j_1j_2}^{k-1}$ and generate a random matrix $\Delta\beta_{j_1j_2}(t_{k})$;
\item Use \eqref{eq:V} to obatin $\mathbf{V}^{k}$,  set $\mathbf{C}^{k}=\mathbf{HV}^{k}\mathbf{H}^T$ and $u_N^{k}=\sum\limits_{m,n=0}^{N-2} c_{mn}^{k}\phi_{m}(x)\phi_n(y)$;
\item Go to the next step.
\end{enumerate}

%We emphasize that we are able to find the explicit form of $\mathbf{A}$ and $\mathbf{B}$ for $A=\Delta$. Obviously, the algorithm can be directly extended to deal with three-dimensional case. On the other hand,  if $-A$ is a general elliptic operator, one can resort to numerical quadrature to compute $\mathbf{A}$ and $\mathbf{B}$, and apply the above procedure. 

\section{Numerical Experiments}
In this section,  two numerical experiments are provided to illustrate the theoretical results claimed in the previous sections.  

%We shall use the following lemma to estimate the growth of eigenvalues of $-\Delta$ on $\mathcal{O}$ (cf. \cite{MaxLW19}).
%\begin{lemma}
% Suppose $\lambda_j$ denotes the j-th eigenvalue of the Dirichlet boundary problem for the Laplacian operator $-\Delta$ in $\mathcal{O}$. Then, we have,
% $$
% C_0 j^{\frac{2}{d}}\leq \lambda_j\leq C_1j^{\frac{2}{d}},
% $$
% where $C_0$ and $C_1$ are independent of $j$. 
%\end{lemma}

\begin{example}
Consider the following 1-d stochastic Allen-Cahn equation on the time domain $0\leq t\leq 1$:
 \begin{equation}
\begin{cases}
\displaystyle{du=\frac{1}{\pi^2}\frac{\partial^2u}{\partial x^2}dt+(u-u^3)dt+g(u)dW^Q(t), \;\; x\in I=(0,1),}\\
u(t,0)=u(t,1)=0,\\
u(0,\cdot)=\sin\pi x
\end{cases}
\end{equation}
and we take 
\begin{align*}
W^Q(t)=\sum\limits_{j=1}^\infty \sqrt{q_j} \sin(j\pi x) \beta_j(t),
\end{align*}
%where $\phi_j(x)=(L_{j-1}(x)-L_{j+1}(x))/\sqrt{4j+2}$. 
Here, $L_j(x)$ is the shifted Legendre polynomials on $[0,1]$ with $q_j$ to be specified below. 
\end{example}

Obviously, eigenfunctions of $A=\frac{\partial^2}{\partial x^2}$ with homogeneous Dirichlet boundary condition on $I$ are $\{\sin j\pi x\}_{j=1}^\infty$,  
and $A$ and $Q$ commute for this case. 
To measure the spatial  error, we run $K = 200$ independent realizations for each spatial expansion terms with $N= 12,14,16,18, 20$ and temporal steps $\tau=1E-5$ and truncate the first $100$ terms in $W^Q(t)$. Since the true solution is unknown, we 
take the numerical solution with $\tau=1E-5$ and $N=100$ as a surrogate. The error $E\|U_N^k-u(t_k,\cdot)\|$ is approximated by
\begin{align}\label{numer:err}
E\|u(t_k,\cdot)-u_N^k\|\approx \sqrt{\frac{1}{K}\sum\limits_{i=1}^K \|u_N^k(\omega_i)-u(t_k,\omega_i)\|^2}.
\end{align}

 First, we consider additive noise and take $g(u)=I$. Hence, we examine the condition $\|A^{\frac{\gamma-1}{2}}Q^{\frac{1}{2}}\|_{L^2}<\infty$ associated with $q_j$ and $\gamma$, and consider the following two cases:

1) $q_j=j^{-1.001}$,  associated with $\gamma=1$;   

2) $q_j=j^{-5.001}$,  associated with $\gamma=3$.

\noindent{One  observes from  Fig \ref{fig1}  that the spatial error decays at a rate of $\mathcal{O}(N^{-\gamma})$  for both cases as Theorem \ref{prop:full} predicts, and the restriction $\gamma<2$ is lifted in contrast to \cite{CuiH19, Wang17, QiW19, KovacsLL15, Kruse14}.  }

Similarly, In order to find the temporal error convergence rate, we freeze $N=100$ and split the time interval $[0,1]$ into $96$, $144$, $192$, $256$, $384$ subintervals for 1) and $256$, $384$, $768$, $1152$, $1536$ for 2), and truncate the first $100$ terms in $W^Q(t)$.   A surrogate of true solution is obtained using $N=100$ and $M=9216$. Fig \ref{fig2} demonstrates that the temporal error decays at a rate of $\mathcal{O}(\tau^{\min\{\frac{\gamma}{2},1\}})$. 

Secondly, in order to demonstrate the prediction in Theorem \ref{prop:full}, we also choose $g(u)=\frac{1-u^2}{1+u^2}$ and $q_j=j^{-5.001}$ in $W^Q(t)$ and repeat the process above. From Fig \ref{fig3}, It is evident that the convergence rate is $\mathcal{O}(N^{-3}+\tau^{1/2})$, which is consistent with Theorem \ref{prop:full}. 

\begin{figure}[ht]
{}\vskip -1.5in
\centering\resizebox{140mm}{150mm}
{\includegraphics{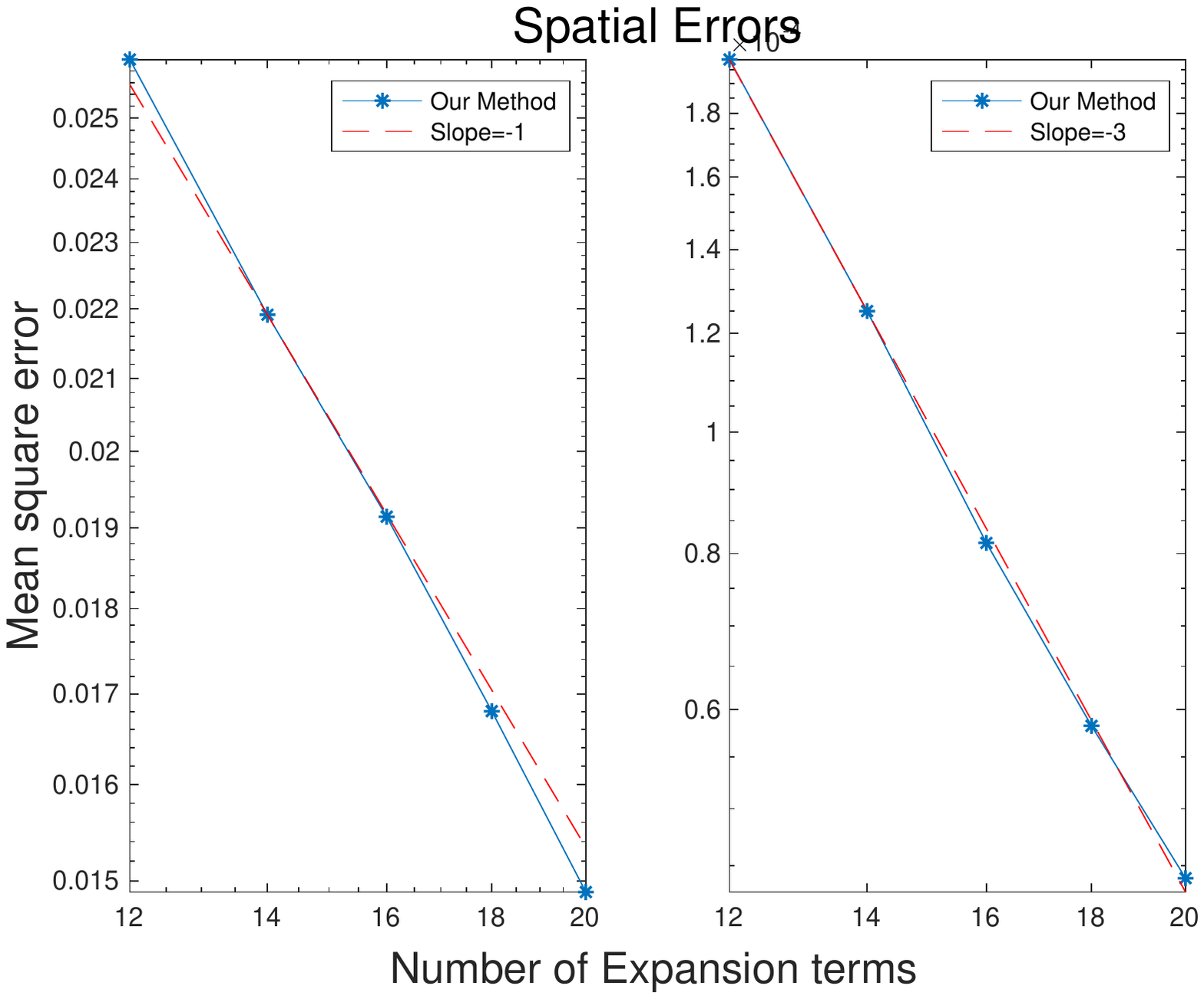}}
\vskip -1.5in
\caption{Spatial errors of 1-d stochastic Allen-Cahn equation with $g(u)=I$: (left) $q_j=j^{-1.001}$ and (right) $q_j=j^{-5.001}$. } \label{fig1}
\end{figure}

\begin{figure}[ht]
{}\vskip -1.5in
\centering
\resizebox{140mm}{150mm}{\includegraphics{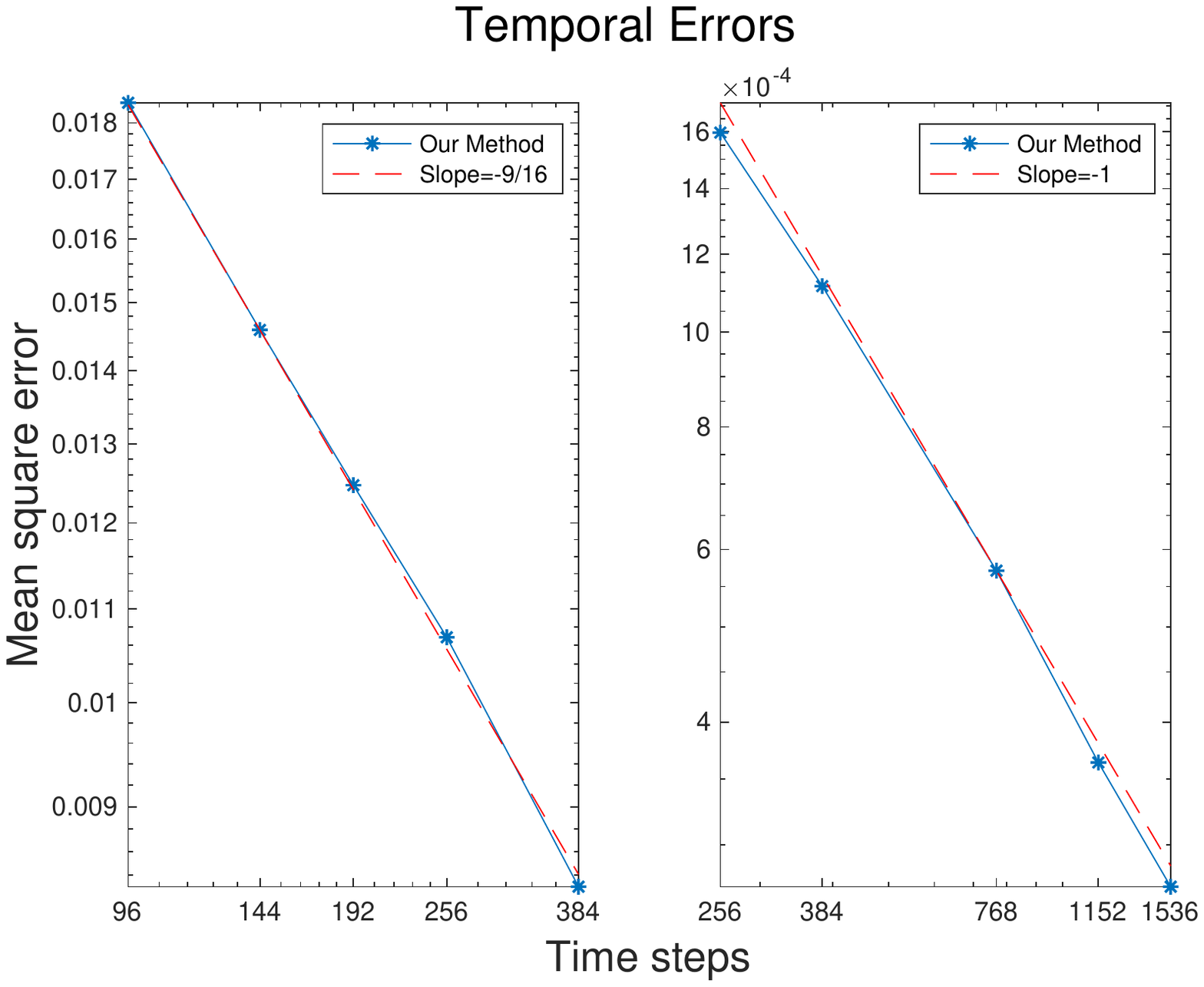}}
\vskip -1.5in
\caption{Temporal errors of 1-d stochastic Allen-Cahn equation with $g(u)=I$: (left) $q_j=j^{-1.001}$ and (right): $q_j=j^{-5.001}$.} \label{fig2}
\end{figure}

\begin{figure}[ht]
{}\vskip -1.5in
\centering
\resizebox{140mm}{150mm}{\includegraphics{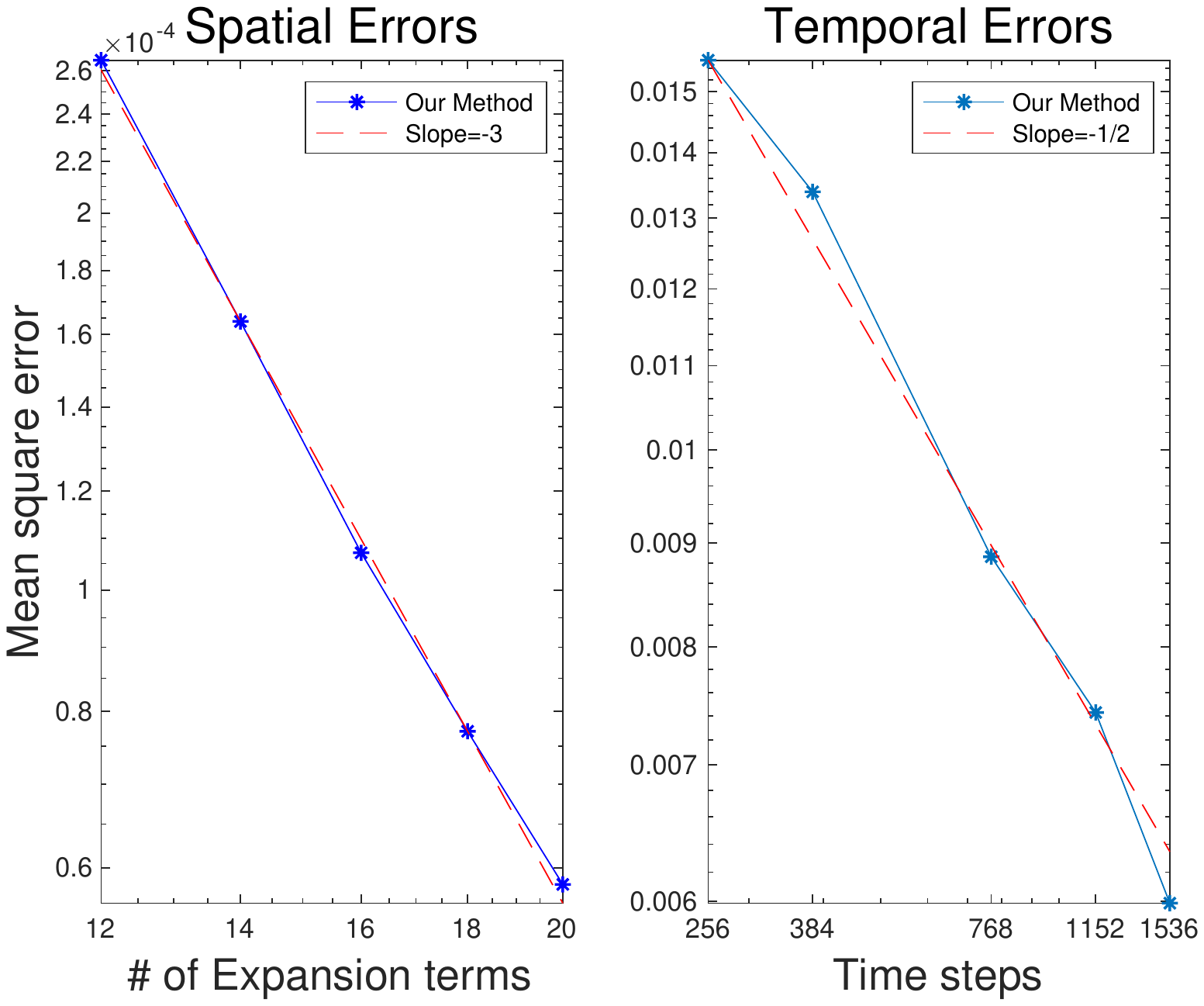}}
\vskip -1.5in
\caption{Numerical errors of 1-d Allen-Cahn equation with $g(u)=(1-u^2)/(1+u^2)$.} \label{fig3}
\end{figure}

\begin{example}
Consider the following 2-d stochastic Allen-Cahn equation:
\begin{equation}
\begin{cases}
du=\frac{1}{2}\Delta udt+(u-u^3)dt+g(u)dW^Q, (x,y)\in (0,1)^2,\notag\\
u_0(x,y,0)=\sin(\pi x)\sin(\pi y),
\end{cases}
\end{equation}
where $g(u)=\sin(u)$   and 
$$
W^Q(t)=\sum\limits_{j_1,j_2=1}^\infty 1/\sqrt{(j_1^2+j_2^2)^3}(\sin (j_1\pi x+\phi_{j_1}(x)))(\sin(j_2\pi y)+\phi_{j_2}(y))\beta_{j_1,j_2}(t).
$$
Here, $\phi(x)$ is defined in \eqref{eq:phim},
\end{example}
%Since $A$ and $Q$ do not commute and the noise is multiplicaitive, we are not able to identify the $\gamma$ which satisfies our assumptions. 

In the experiment,  we choose $K=200$ in \eqref{numer:err} to measure the error again.  To balance the CPU runtime and  accuracy, we truncate the first $10$ terms in each direction of $W^Q(t)$. In order to find the spatial convergence rate, we use the numerical solution with $N=100$ and $M=1000$ for $T=0.1$ as a surrogate of true solution. 
From Figure \ref{fig4}, we can clearly observe a spatial convergence rate of approximately $\mathcal{O}(N^{-3/2})$ for $N=16,17,19, 21$ and $22$. 

Similarly, in order to find temporal convergence rate, we use the numerical solution with $N=60$ and $M=2304$ for $T=0.5$ as a surrogate of true solution. It is clear that temporal convergence rate
$\mathcal{O}(\tau^{1/2})$ for $\tau=1/64, 1/72, 1/96, 1/128, 1/144$ can be observed from Figure \ref{fig4}.  
\begin{figure}[ht]
{}\vskip -1.5in
\centering
\resizebox{140mm}{150mm}{\includegraphics{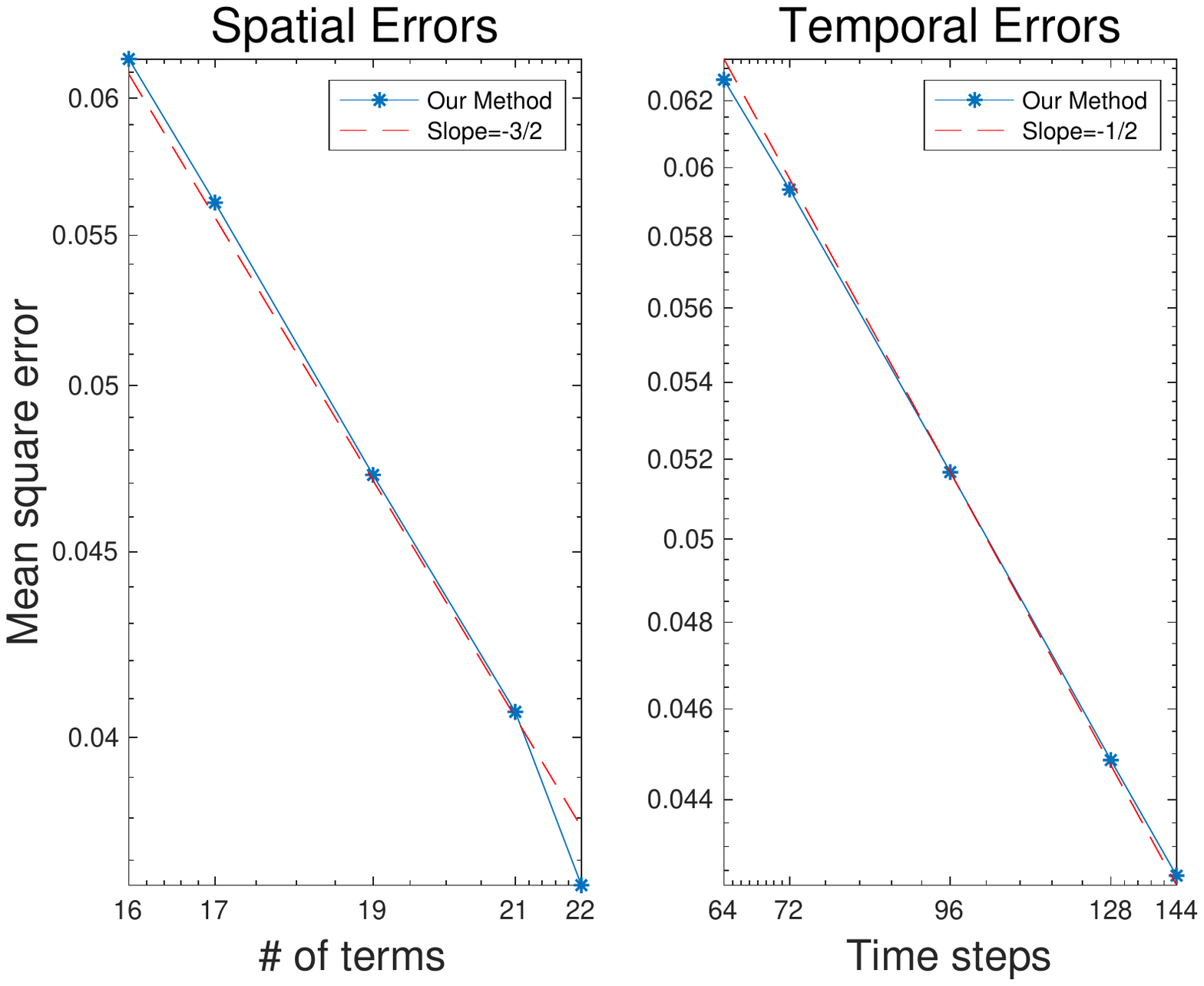}}
\vskip -1.5in
\caption{Numerical errors of 2d stochastic A-C equation.} \label{fig4}
\end{figure}

\section{Concluding remarks}
We developed  a fully discrete scheme   for nonlinear stochastic partial differential equations with non-globally Lipschitz coefficients driven by multiplicative noise in a multi-dimensional setting. The space discretization is a Legendre spectral method, so it  does not require the elliptic operator $A$ and the covariance operator $Q$ of noise in the equation commute, %and thus successfully alleviates a restriction of Fourier spectral method for SPDEs pointed out by Jentzen, Kloeden and Winkel  in \cite{JentzenKW11} 
while can still be efficiently implemented as with a Fourier method.
%[Efficient simulation of nonlinear parabolic SPDEs with additive noise, Ann. Appl. Probab. 21(3): 908-950, 2011].
 The time discretization is a tamed semi-implicit scheme which treats the nonlinear term explicitly while being unconditionally stable, and it avoids solving nonlinear systems at each time step.  
Under reasonable regularity assumptions, we established optimal strong convergence rates in both space and time for our fully discrete scheme.  We also presented several numerical experiments to validate   our theoretical results.  

We note that the   fully discrete scheme constructed in this paper can also be used for the three-dimensional case, but our error analysis is only valid for   one- and two-dimensional cases.

\end{document}